\documentclass[12pt]{article}
\usepackage{latexsym,amssymb,amsmath,amsfonts,amsthm}
\usepackage[usenames,dvipsnames]{color}
\usepackage{bm}
\definecolor{purple}{rgb}{0.9,0,0.8}

\newtheorem{theorem}{Theorem}

\newtheorem{lemma}{Lemma}

\newtheorem{conj}{Conjecture}

\def\beq{ \begin{equation} }
\def\eeq{ \end{equation} }
\def\mn{\medskip\noindent}
\def\ms{\medskip}

\def\ep{\epsilon}

\def\square{\vcenter{\vbox{\hrule height .4pt
  \hbox{\vrule width .4pt height 5pt \kern 5pt
        \vrule width .4pt} \hrule height .4pt}}}
\def\eopt{\hfill$\square$}

\def\xyz#1{}
\def\clearp{}

\usepackage{graphicx}
\graphicspath{%
    {converted_graphics/}
    {/}
}
\begin{document}

\title{The Evolving Voter Model on Thick Graphs}
 \author{Anirban Basak, Rick Durrett, and Yuan Zhang 
\thanks{RD and YZ were partially supported by NSF grant DMS 1305997. YZ was a postdoc at UCLA
in 2015--2016, and now has a three-year postdoc at Texas A\&M.}\\
Dept.~of Math, Duke U., Durham, NC 27708-0320}
\date{\today}

\maketitle

\begin{abstract}
In the evolving voter model, when an individual interacts with a neighbor having an opinion different from theirs, they will
with probability $1-\alpha$ imitate the neighbor but with probability $ \alpha$ will sever the 
connection and choose a new neighbor at random (i) from the graph or (ii) from those with the same opinion. Durrett et al.~\cite{evo8} used simulation and heuristics to study these dynamics on sparse graphs.
Recently Basu and Sly \cite{BaSly} have analyzed this system  with $1-\alpha = \nu/N$ on a dense Erd\H{o}s-R\'{e}nyi graph $G(N,1/2)$ and
rigorously proved that there is a phase transition from rapid disconnection into components with a single opinion
to prolonged persistence of discordant edges as $\nu$ increases.
In this paper, we consider the intermediate situation of Erd\H{o}s-R\'enyi random graphs with average degree $L=N^a$ where $0 < a < 1$.
Most of the paper is devoted to a rigorous analysis of an approximation of the dynamics called the approximate master equation.
Using ideas of \cite{LMR} and \cite{Silk} we are able to analyze these dynamics in great detail. 
\end{abstract}  

\section{Introduction}

We consider a simplified model of a social network in which individuals have one of two opinions (called 0 and 1) and their opinions and the network connections coevolve. In the discrete time formulation, oriented edges $(x,y)$ are picked at random. If $x$ and $y$ have the same opinion no change occurs.
If $x$ and $y$ have different opinions then: with probability $1-\alpha$, the individual at $x$ imitates the opinion of the one at $y$; otherwise, i.e., with probability $\alpha$, the link between them is broken and $x$ makes a new connection to an individual $z$ chosen at random (i) from those with the same opinion (``rewire-to-same''), or (ii) from the network as a whole (``rewire-to-random''). The evolution of the system stops when there are no longer any ``discordant'' edges that connect individuals with different opinions. 

Holme and Newman \cite{HN} were the first to consider a model of this type. They chose option (i), rewire-to-same, and initialized the graph with large number $K$ of 
opinions so that $N/K$ remained bounded as the number of vertices $N\to\infty$. They argued that there was a critical value $\alpha_c$ so that
for $\alpha>\alpha_c$, the graph rapidly disconnects into a large number of small components while if $\alpha < \alpha_c$, a giant community of 
like-minded individuals of size $O(N)$ formed.

The work of Holme and Newman  \cite{HN} was followed by a number of papers in the physics literature. References can be found in Durrett et al.~\cite{evo8} and Silk et al.~\cite{Silk}.
Recent papers study several variants of the model include \cite{KimHay}, \cite{NKB}, \cite{RogGro}, and \cite{DESM}.
Here we will stick to the basic version. Let $p$ be the initial fraction of voters with opinion 1 and
let $\pi$ be the fraction of voters holding the minority opinion after the evolution stops. Through a combination of simulation and heuristics, Durrett et al.~\cite{evo8} argued that 

\begin{itemize}
  \item 
In case (i), rewire-to-same, there is a critical value $\alpha_c$ which does not depend on $p$, with $\pi \approx p$ for $\alpha > \alpha_c$ and $\pi \approx 0$ for $\alpha < \alpha_c$. 
\item
In case (ii), rewire-to-random, the transition point $\alpha_c(p)$ depends on the initial density $p$. For $\alpha > \alpha_c(p)$, $\pi \approx p$, but for $\alpha < \alpha_c(\rho)$ we have $\pi(\alpha,p) = \pi(\alpha,1/2)$. 
\end{itemize}

\noindent
The graphs in Figures \ref{fig:rts} and \ref{fig:rtr} should help clarify these claims.

\begin{figure}[h] 
  \centering
  \includegraphics[bb=1 2 675 472,height=3.25in,keepaspectratio]{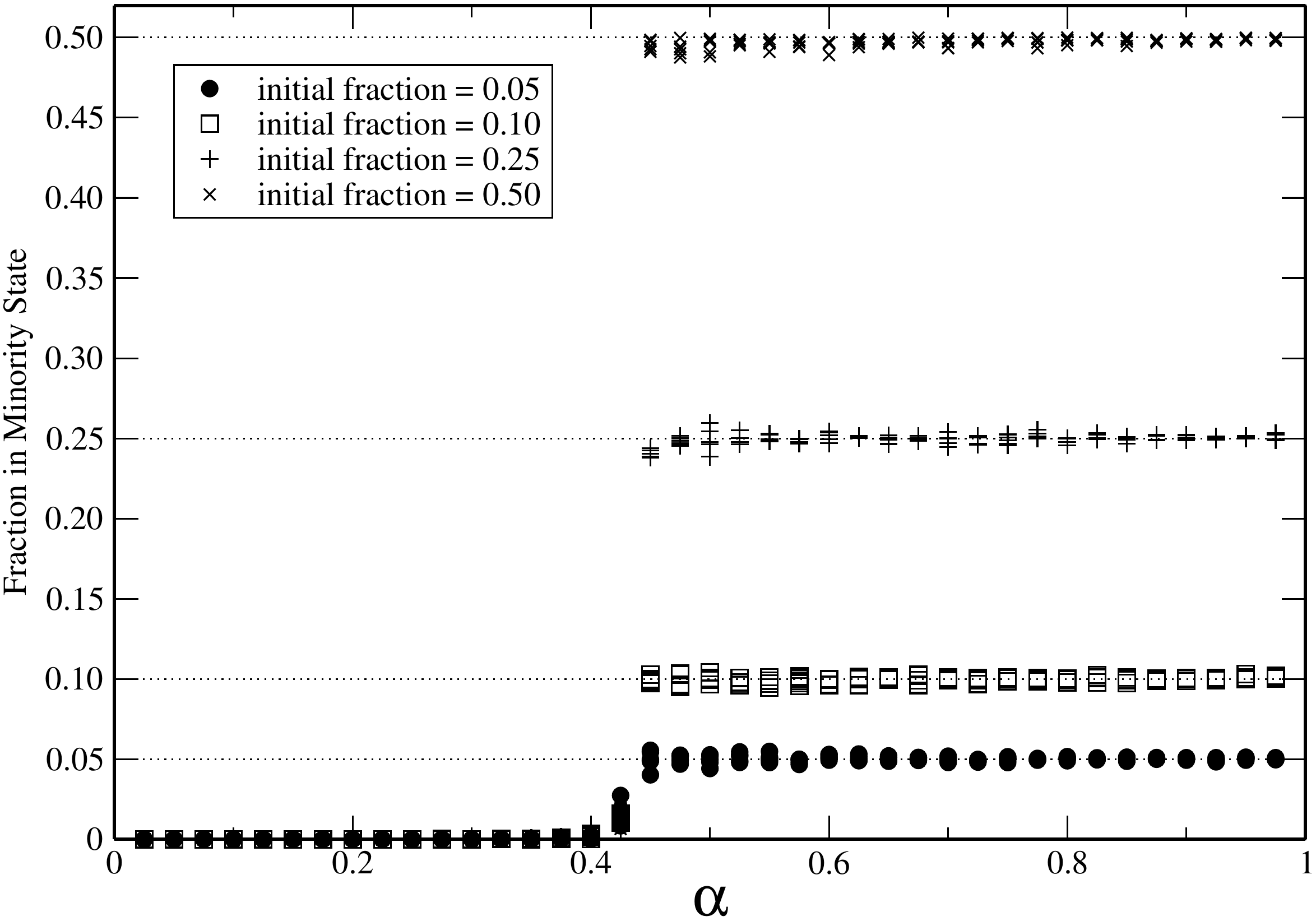}
  \caption{Simulation of rewire to same model for Erd\H{o}s-R\'{e}nyi graphs with 100,000 vertices and average degree 4. We start with an initial 
product measure with density $p= 0.5$, 0.25, 0.1, or 0.05 and vary $\alpha$. As $\alpha$ decreases from 1, the ending density $\pi(p) \approx p$ and then at $\alpha_c \approx 0.42$ it drops to $\pi(p) \approx 0$. }
  \label{fig:rts}
\end{figure}

\begin{figure}[tbp] 
  \centering
  \includegraphics[bb=1 1 675 468,height=3.25in,keepaspectratio]{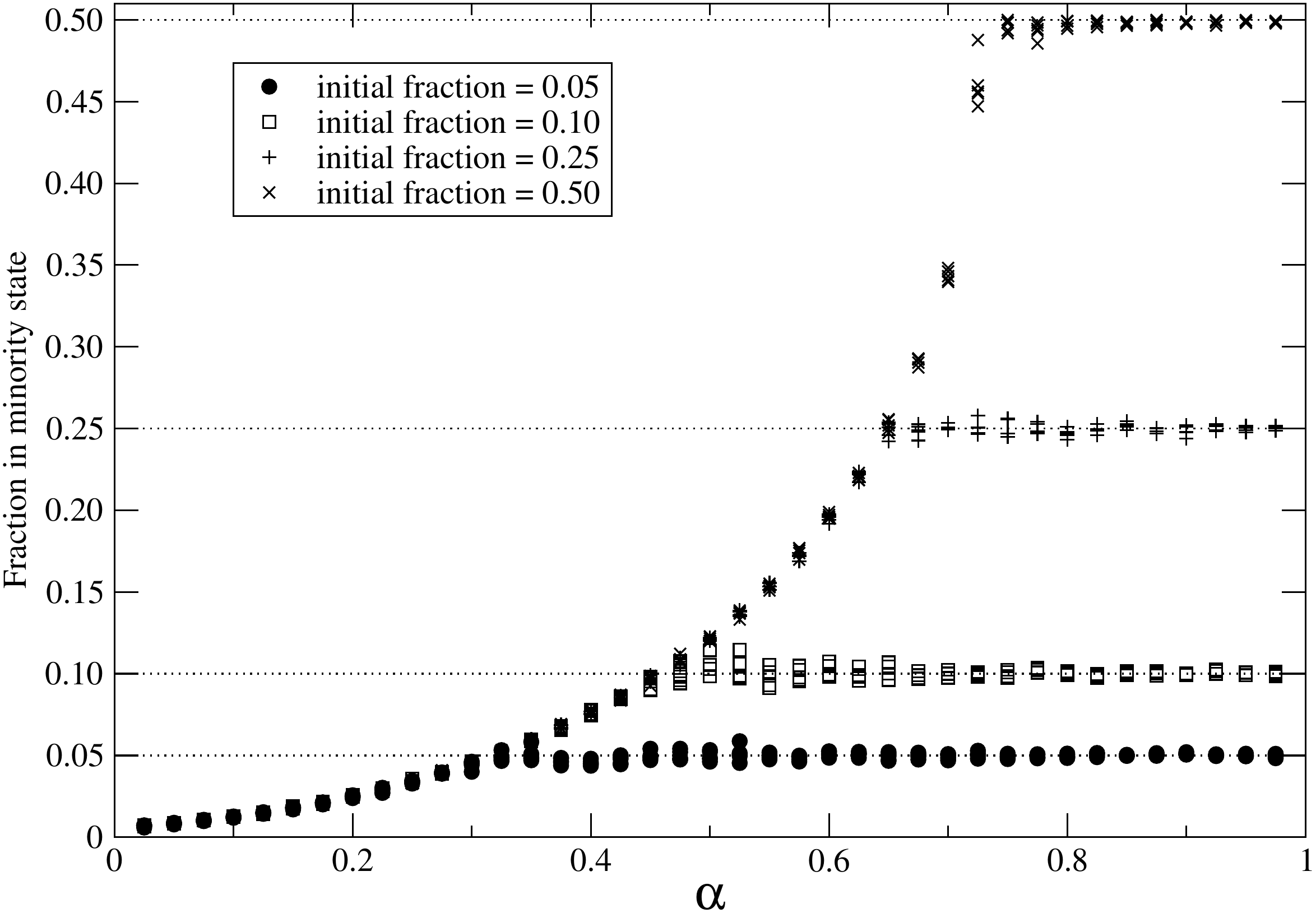}
  \caption{Simulation of rewire to random model for Erd\H{o}s-R\'{e}nyi graphs with 100,000 vertices and average degree 4, starting from 
product measure with densities $p= 0.5$, 0.25, 0.1, or 0.05. In \cite{evo8}, $\alpha \to \pi(\alpha,1/2)$ is called
the universal curve because for $p<1/2$, $\pi(\alpha,p)$ is constant for $\alpha > \alpha_c(p)$ and then follows the universal curve.}
  \label{fig:rtr}
\end{figure}

If we formulate the evolving voter model in continuous time with each oriented
edge subject to updating at rate 1, then arguments in \cite{HN} and \cite{evo8} suggest that for $\alpha > \alpha_c(p)$ the disconnection takes time $O(\log N)$, i.e., $O(N\log N)$ updates, while for $\alpha < \alpha_c(p)$ the time becomes $O(N)$, i.e., $O(N^2)$ updates. The first conclusion is easy to explain: if we
rewire-to-same and $\alpha=1$, then disconnection will occur when all of the edges have been touched. If there are $M$ edges, then by the coupon collectors problem, this requires time $O(M \log M)$ where $M$ is the number of edges. 

The explanation for the long time survival is more complicated and, at the moment, is based on phenomena observed in simulation and not yet rigorously demonstrated. The intuitive picture is motivated by a result of Cox and Greven \cite{CoxG}. To state their result we recall that the voter model on the $d$-dimensional lattice with nearest neighbor interactions has a one parameter family of stationary distributions $\nu_\theta$, indexed by the fraction of sites in state 1.

\begin{theorem}
If the voter model on the torus in $d \ge 3$ with $N$ sites starts from product measure with density $p$ then at time $Nt$ it looks locally like $\nu_{\theta(t)}$ where the density $\theta(t)$ changes according to the Wright-Fisher diffusion process
$$
d\theta_t = \sqrt{\beta_d \cdot 2\theta_t(1-\theta_t)}
$$
and $\beta_d$ is the probability that two random walks starting from neighboring sites do not hit.
\end{theorem}

\noindent
In words, this is true because there is a separation of time scales: 

\mn
($\star$) {\it The time to converge to equilibrium is much smaller than the time needed for the density to change, so if time is scaled appropriately then the system is always close to an equilibrium and the parameter follows a diffusion process. }

\ms
Let $N_1(t)$ be the number of vertices in state 1 at time $t$. Durrett et al.~\cite{evo8} demonstrate that ($\star$) is true
for the evolving voter model by plotting various statistics versus $N_1(t)$ and showing that values were close to a curve, i.e., the values of all of the statistics are determined by $N_1(t)$. That is, there is a one parameter family of quasi-stationary distributions and the densities change slowly over time. Simulations supporting this claim for the evolving voter model on sparse graphs can be found in \cite{evo8}. Here, we will present simulation results for the version of the model in which the average degree of vertices is $L$ and the voting rate $1-\alpha = \nu/L$. (We will describe the system in more detail in the next section.) 

\clearp

\begin{figure}[h] 
  \centering
  \includegraphics[height=2.8in,keepaspectratio]{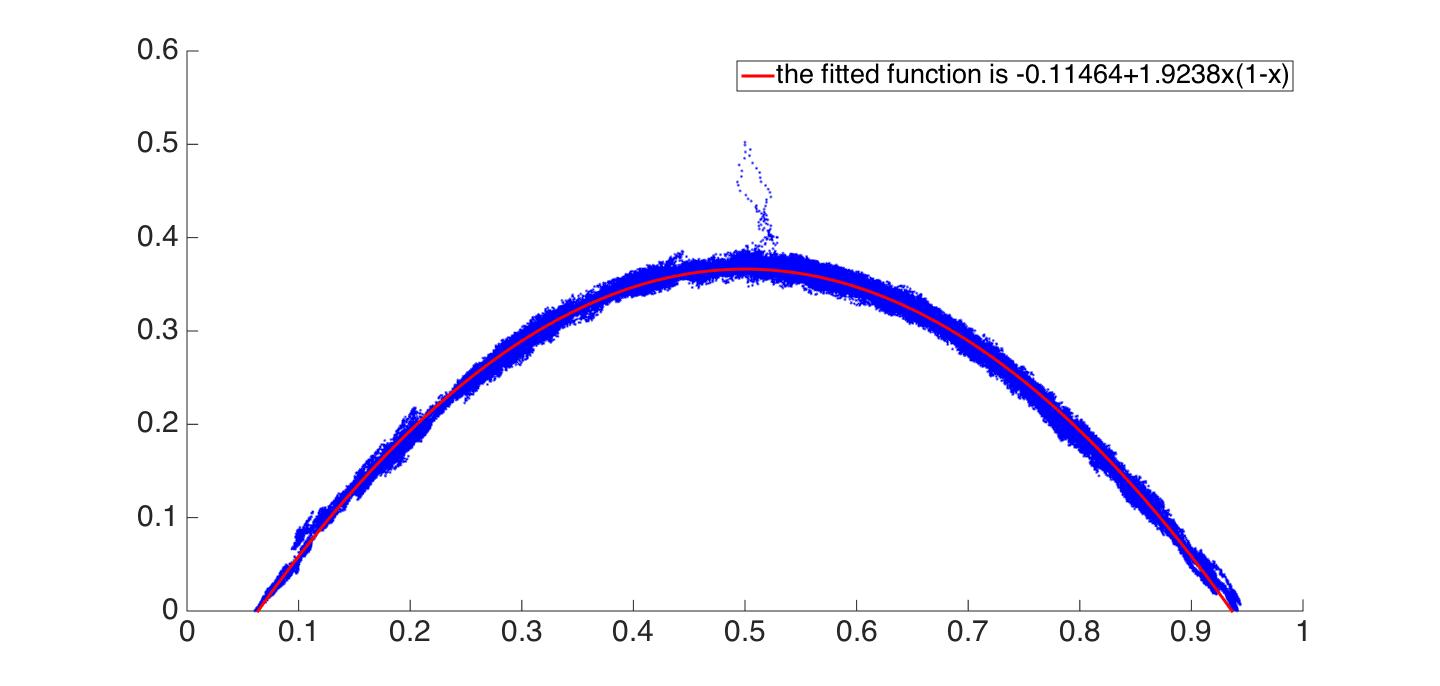}
  \caption{Plot of $N_{10}$ versus $N_1$ when $N=2500$, $L=50$, $\nu = 2.5$}
  \label{fig:2500vs50}
\end{figure}

\begin{figure}[h] 
  \centering
  \includegraphics[height=2.8in,keepaspectratio]{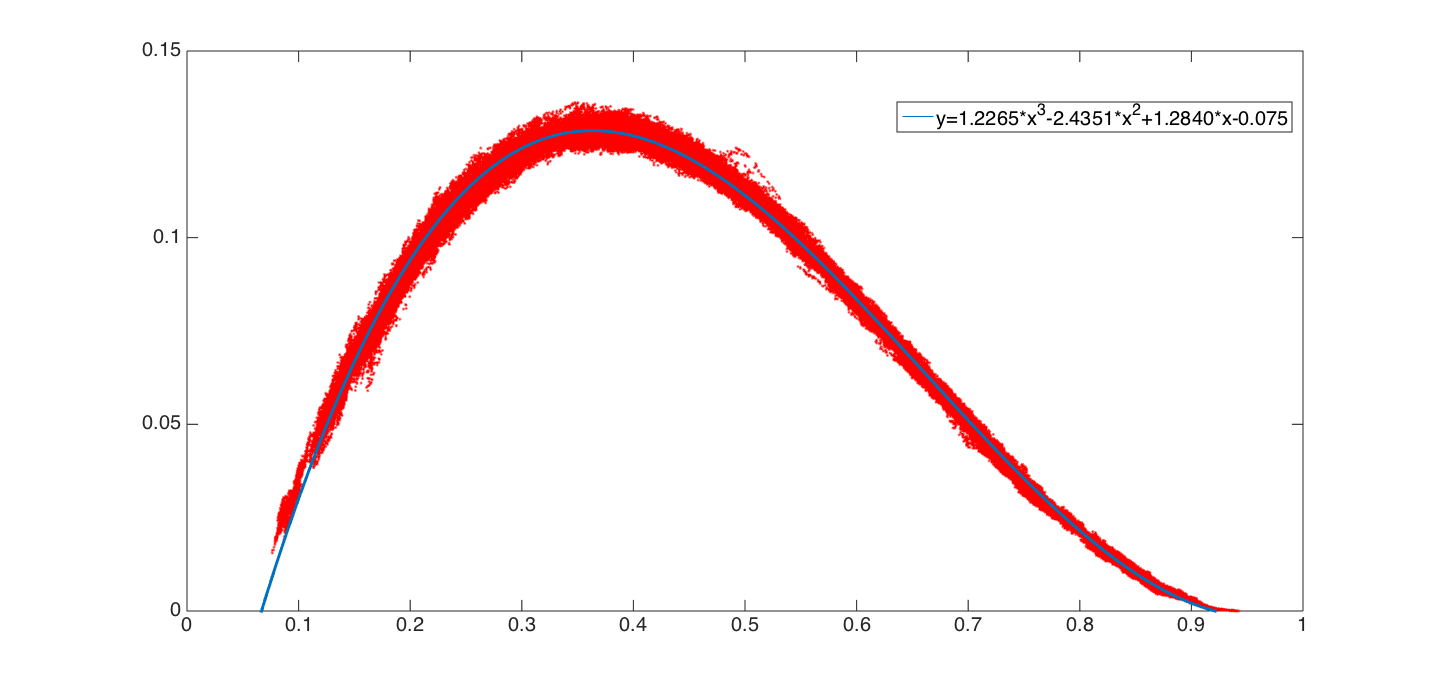}
 \caption{Plot of $N_{100}$ versus $N_1$ when $N=2500$, $L=50$, $\nu = 2.5$}
  \label{fig:N001_fit}
\end{figure}

\noindent
Figure \ref{fig:2500vs50} gives a simulation of the system with $N=2500$, $L=50$, $\nu=2.5$ and shows that the $(N_1(t),N_{10}(t))$ is well approximated by the quadratic equation $1.9238x(1-x)-0.11464$. Following \cite{evo8} we call this curve the ``arch.'' Let $N_{100}$ be the number of $(x,y,z)$ in the graph so that $y$ is a neighbor of $x$, $z$ is a neighbor of $y$ and the states of $x,y,z$ are $1,0,0$. Figure \ref{fig:N001_fit} plots the number of $N_{100}(t)$ in the graph versus $N_1(t)$. Again the values are close to a curve indicating the statistic $N_{100}$ is determined by $N_1$. This time the fitted curve is a cubic, which has the same zeros as the quadratic.

To describe the implications of this picture for the (conjectured) behavior of the process, we note that the fitted quadratic in Figure \ref{fig:2500vs50} has roots at 0.0737 and 0.9263, as does the cubic. If we start from $p=1/2$, then the system rapidly comes to a quasi-stationary distribution $\mu_{1/2}$. On the time scale $Nt$ it  is close to $\mu_{\theta(t)}$ until the value of the parameter reaches one of the endpoints of the ``arch"where $N_{10}=0$ and disconnection occurs. When $N$ is large the initial density will not change significantly at times $o(N)$, i.e., $o(N^2L)$ updates (this will be proved later, see Section \ref{sec:EEPA}) . Thus we expect the same final behavior if the initial $p \in (0.0737,0.9263)$, while if $p$ is outside the interval then rapid disconnection occurs. 

Figure \ref{fig:2500vs50b} gives a simulation of $N=2500$, $L=50$, $\nu=1$. The arch is now smaller with endpoints at roughly 0.3 and 0.7. If $\nu > \nu_c(1/2)$ and we let $a(\nu),1-a(\nu)$ be the endpoints of the arch  then
$$
\nu_c(p) = \inf\{ \nu : p \in (a(\nu),1-a(\nu)) \}.
$$
By arguments in the last paragraph when $\nu < \nu_c(p)$ rapid disconnection occurs, while if $\nu > \nu_c(p)$ the ending minority fraction is the same as if we started from $p=1/2$. Changing variables $\alpha = 1  - \nu/L$, we see that the intuitive picture agrees with the behavior shown in Figure \ref{fig:rtr}. As explained in \cite{evo8} the behavior in the case of rewire to same shown in Figure \ref{fig:rts} is due to the fact that when the arch exists it end points are always 0 and 1. See Figure 8 in that paper.

\begin{figure}[h] 
  \centering
  \includegraphics[width=5.67in,height=2.7in,keepaspectratio]{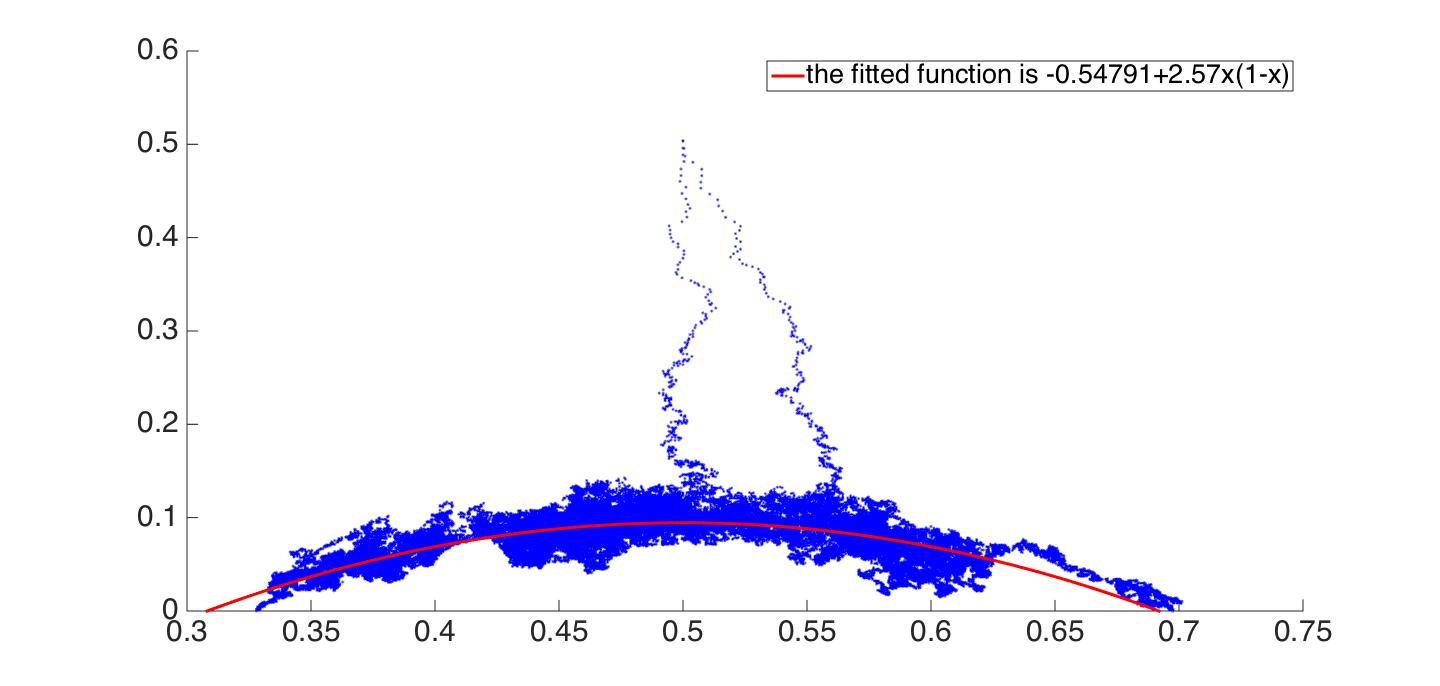}
  \caption{Plot of $N_{10}$ versus $N_1$ when $N=2500$, $L=50$, $\nu = 1$}
  \label{fig:2500vs50b}
\end{figure}

The rest of the paper is devoted to different approaches to analyzing the evolving voter model. In Section \ref{sec:BasuSly} we describe the recent results of Basu and Sly \cite{BaSly} for the process on $G(N,1/2)$ and we extend two of their results to the case of thick graphs. In each case the bounds do not depend on $L$ supporting our conjecture that $\nu_c$ does not depend on $L$. Proofs are deferred to Section \ref{sec:BSth}.

In Section \ref{sec:EEPA}, we provide exact equations for the evolution of finite-dimensional distributions in the model: e.g., $N_{ij} =$ the number of ordered pairs of adjacent sites in state $i$ and $j$, and $N_{ijk} =$ the number of ordered triples of adjacent sites. Derivation of these equations are deferred to Section \ref{sec:dereq}. As is often the case in interacting particle systems equations for $k$ site probabilities involve $k+1$ site probabilities so these cannot be solved. One approach is to use the pair approximation to express three site probabilities in terms of two site probabilities to obtain a closed system. When we began this work we thought (or hoped) that since $L\to\infty$ the pair approximation would give the right answer. As we explain in Section 3.1, this is not true for the simple reason that the second moment is larger than the square of the mean. 

In Section \ref{sec:AME}, we introduce the approximate master equation (AME) for $A_{k,\ell}(t)$ the number of sites in state 1 at time $t$ that have $k$ neighbors in state 1 and $\ell$ in state 0, and $B_{k,\ell}$ the number of sites in state 0 that have $k$ neighbors in state 1 and $\ell$ in state 0. Since we do not know the state of the neighbors of the neighbors, we use ratios of known probabilities to find their distribution. For example, we use $N_{101}/N_{10}$ to estimate the number of 1 neighbors of a 0 that is adjacent to a 1, in contrast to the pair approximation which declares that this is always equal to $N_{01}/N_0$. If we view a particle in state $i$ with $k$ neighbors in state 1 and $\ell$ in state 0 as a point at $(k,\ell)$ in plane $i$. This leads to an interesting system where vertices walk around in two planes (one for those in state 1, the other for those in state 0) and jump to the other plane when voter events change their states. Using recent results of Lawley, Mattingly, and Reed \cite{LMR} we can prove this system converges in distribution as $t\to\infty$. 

In Section \ref{sec:GFPDE}, we use an approach of Silk et al.~\cite{Silk} to derive properties of the limiting distribution. We write a pair of partial differential equations for the generating functions 
$Q(t,x,y) = \sum_{k,\ell} A_{k,\ell}(t) x^k y^{\ell}$ and $R(t,x,y) = \sum_{k,\ell} B_{k,\ell}(t) x^k y^{\ell}$, scale the degrees by $L$ and then take the limit $L\to\infty$ to arrive at PDEs for the limiting generating functions $U(a,b)$ and $V(a,b)$ for the equilibrium distribution, see \eqref{Ueqr} and \eqref{Veqr}.

In Section \ref{sec:phalf} we postulate power series solutions for the generating functions and study the symmetric case $p=1/2$. {As is often the case in interacting particle systems,} there are not enough equations to find the coefficients but we are able to compute $U_{aa}$, $U_{ab}$ and $U_{bb}$ from $U_b$ (which is the expected number of $1,0$ edge). If we use simulations to find $U_b$ then the predicted values $U_{ab}$ and $U_{bb}$ are off by only 1\% while the predicted value of $U_{aa}$, the number of $111$'s, is off by 10\%.

The total number of edges is not conserved in the AME. Our initial goal was to find self-consistent solutions to the AME. That is, values of the five parameters given in \eqref{grdef} that result in an equilibrium in which the statistics agree with parameters. Silk et al.~\cite{Silk} carry this out for the rewire to same dynamics but our computer skills do not allow us to replicate their computation. We would try harder if {it was possible to use the method for the asymmetric case.} 

The remainder of the paper is devoted to proofs.
Section \ref{sec:dereq} provides the derivation of equations for the evolution of the ``finite-dimensional distributions'' in the evolving voter model. To derive these equations, we fix a pair of adjacent vertices $x$, and $y$. Then we consider all possible cases of voting and rewiring that can change the states, or the connectivity between those two vertices. Summing over all possible pairs gives us the desired equations. 

From the exact equations, we can obtain the pair approximation, a closed system of equations. In Section \ref{sec:PT1}, using these modified equations we can find the average number of neighbors of a site in state $i$ that are in state 1, $J_i$, and in state 0, $K_i$. These computations make a prediction about the critical value, which simulation shows is incorrect. However, the pair approximation value might be a lower bound on the true critical value.

Section \ref{sec:ODE} and Section \ref{sec:Ctoeq} deal with the analysis of approximate master equations. As we have mentioned, AME transforms the evolving voter model into a system of $N$ particles moving on two planes and jumping between them. In Section \ref{sec:ODE} we analyze the differential equations for each individual plane, and show that they have globally attracting fixed points. Building on this, in Section \ref{sec:Ctoeq} we show that the two plane system has a unique stationary distribution. Here the results of Lawley, Mattingly, and Reed \cite{LMR} are helpful. Although, their set-up does not quite match ours, we can adapt their techniques to make it work in our case. Standard techniques from renewal theory, and ergodic theory helps here.

In Section \ref{sec:derPDE} we derive the differential equations corresponding to the generating functions of $A_{k,\ell(t)}$, and $B_{k,\ell}(t)$. Again, this follows upon a careful consideration of all possible cases of voting, and rewiring. Section \ref{sec:momeq} provides the derivation of the moment equations corresponding to the limiting generating functions $U$ and $V$ in the symmetric case $p=1/2$.

In Section \ref{sec:BSth}, we provide the proofs of extensions of two result of Basu and Sly to thick graphs. Our improvements in their proofs are minor. Our Lemma \ref{Dmax} gives a better control on the maximum degree of the evolving graph after an amount of time $O(NL)$, which helps us to obtain better bound on the threshold for both the theorems (see the statements  of Theorem \ref{NL1i} and Theorem \ref{NL2} in Section \ref{sec:BasuSly}).

\clearp

\section{Results of Basu and Sly}\label{sec:BasuSly}

Recently Basu and Sly \cite{BaSly} have rigorously proved the existence of a phase transition for the dynamics described
above on the dense Erd\H{o}s-R\'{e}nyi graph $G(N,1/2)$. They work in
discrete time with voter events occurring with probability $1-\alpha = \nu/N$. They prove three results. To state them we need some notation.
Let $\tau$ be the first time there are no discordant edges.
Let $N_*(t)$ be the number of vertices holding the minority opinion at time $t$ and for $0 < \ep < 1/2$ let $\tau_*(\ep) = \min\{ t: N_*(t) \le \ep n\}$.

In all three results stated here the system starts from product measure with density 1/2.
In their first result, they use the efficient version of the model in which only discordant edges are chosen at random for updating. 

\begin{theorem} \label{BS1i}
There is a $\nu_0$ so that for all $\nu < \nu_0$ and any $\eta>0$
$$
P( \tau < 10N^2,\ N_*(\tau) \ge \frac{1}{2} - \eta ) \to 1 \quad\hbox{as $N\to\infty$}.
$$
\end{theorem}

\noindent
The number of edges $M \sim N(N-1)/2$.
Separation requires $O(N^2)$ updates rather than $O(M \log M)$ because the efficient algorithm always picks discordant edges while the one with random choices takes a long time to find the last few remaining discordant edges. The second result says that the density of voters with opinion 1 does not change much from its initial value of 1/2. Since the number of voters with opinion 1 is a martingale this follows easily once rapid disconnection is established. 

They prove their second and third results for the discrete time algorithm in which edges are chosen at random.
The next theorem is the main result of their paper and has a very long and difficult proof. 

\begin{theorem}\label{BS1ii}
Let $\ep' \in(0,1/2)$ be given. There is a $\nu_*(\ep')$ so that for $\nu > \nu_*(\ep')$
we have $\tau_*(\ep') \le \tau$ with high probability and
$$
\lim_{c\downarrow 0} \liminf_{N\to\infty} P( \tau > cN^3 ) = 1.
$$
\end{theorem} 

\noindent
If each edge were chosen at rate 1, then $O(N^3)$ updates translates into time $O(N)$ and is consistent with the results
stated in the previous section.

Theorem \ref{BS1ii} gives a lower bound on the disconnection time and shows that if $\nu$ is large then before disconnection 
occurs the minority fraction has been $\le \ep'$. The next result shows that in the rewire-to-random case for fixed $\nu$ there is a lower bound on the minority fraction when fixation occurs. This is 
consistent with the simulation for sparse graphs shown in Figure \ref{fig:rtr},
but is believed to be false for rewire to same on sparse graphs, see Figure \ref{fig:rts}.

\begin{theorem} \label{BS2}
Let $\nu >0$ be fixed. For the rewire-to-random model, there is an $\ep_*(\nu)$ so that $\tau < \tau_*(\ep_*)$ with high probability.
\end{theorem}

\subsection{Results for thick graphs}

The evolving voter model on a dense Erd\H{o}s-R\'enyi random graph $G(N,1/2)$ is ugly because it will quickly develop self-loops and parallel edges.
To avoid this problem, while retaining the simplifications that come from having vertices of large degree, we will consider Erd\H{o}s-R\'enyi random
graphs in which the mean degree is $L$ with $L=N^a$ and $0 < a < 1$. This regime is intermediate between dense graphs with
$L=O(N)$ and sparse graphs with $L=O(1)$, so we call them {\it thick graphs}.
Since a Poisson distribution with mean $L$ has
standard deviation $\sqrt{L}$ there is little loss of generality in supposing that we start with a random graph in which each vertex has 
degree $L$. To do this, we have to assume $LN$ is even.

Following Basu and Sly, voting occurs on each oriented edge $(x,y)$ at rate $\nu/L$, i.e., $x$ imitates $y$;
while at rate 1, $x$ severs its connection to $y$ and connects to a randomly chosen vertex $z$ that is not already
one of its neighbors. We can drop the $-\nu/L$ from the rewiring rate since $\nu/L\to 0$, but in this
section we will retain it to have a closer connection with \cite{BaSly}. 

Theorems \ref{BS1i} and \ref{BS2} generalize in a straightforward way to the new model. 
Here, and throughout the paper, we will consider only the rewire-to-random version
and let $p$ be the initial fraction of vertices in state 1.
In the next two results, we consider discrete time and use the efficient algorithm in
which at each step a discordant edge is selected for updating. Theorem \ref{BS1i} becomes

\begin{theorem} \label{NL1i}
Suppose $p \le 1/2$ and let $\ep>0$. If $\nu \le (0.15)/(1+3p)$ then with high probability $\tau < 3pNL$,
and at time $\tau$ the fraction of 1's is between $p-\ep$ and $p+\ep$ with high probability.
\end{theorem}

\noindent
Note that the bound does not depend upon the average degree $L$. When $p=1/2$ the bound is 0.06. 

\mn
{\bf Sketch of the proof.}
Here we have followed the proof in \cite{BaSly} with some improvements in the arithmetic. Let $X_m$ be the number of discordant edges after $m$ updates. Independent of the current frequency of sites in state 1, every time a rewiring event occurs $X_m$ decreases by 1 with probability $1/2$ and stays the same with probability 1/2. To handle voting events, we take the drastic approach that they can at most increase $X_m$ by
$D_{max}(m)$, the maximum degree of vertices in the graph, and $D_{max}(tNL) \le (1+\ep + t)L$, {the second bound resulting
from estimating the number of times a vertex is chosen to receive a rewiring, ignoring the fact that vertices will lose neighbors due to rewirings}.
\eopt 

\ms
Our next result is Theorem 6.1 in \cite{BaSly}, from which Theorem \ref{BS2} stated above follows easily. Let 
${\cal G}(p,N,L) =$ graphs with vertex set $V = \{ 1, 2,, \ldots, N \}$ labeled with 1's and 0's so that 
the number of vertices in state 1, $N_1(G)=pn$, and the number of edges is $NL/2$.

\begin{theorem} \label{NL2}
Let $\nu>0$ and $\ep>0$
There is a $p(\nu) < 1/2$ so that for all $G(0) \in {\cal G}(p,N,L)$ with $p \le p(\nu)$, we have with high probability
$\tau < 7NL$  and the fraction of vertices in state 1 at time $\tau$ is $\in (p-\ep,p+\ep)$.
\end{theorem}

\noindent
Since this result assumes nothing about the graph except for the number of edges, it follows that if the density of 1's 
gets to $p(\nu)$ then rapid disconnection will occur. The proof will show that we can take $p(\nu) = (\nu/60) e^{-21\nu}$,
which again is independent of $L$. When $\nu = 1$, $p(\nu) = 1.26 \times 10^{-11}$. Even though the value is tiny the form of the bound allows us to conclude that for $p < p_0$ the threshold for prolonged persistence $\nu_c(p) \ge (1/21) \log(1/p) \to \infty$ as $p \to 0$.

\mn
{\bf Sketch of the proof.} The proof is clever but again uses arguments that are extremely crude. One uses a special construction in which 
counters $K(v,m)$ determine if an event on an oriented edge $(v,u)$ at update $m$ will be a voting
(the counter is 0) or a rewiring. The $K(v,m)$ are initialized to be independent geometrics and 1 is subtracted each time
the vertex is used. 

To study the dynamics, we divide the graph into the set of vertices $S$ with initial degree $\le 11L$ and $T = V-S$. The key observation is that if a 0 in $S$ is changed to a 1 by voting and the counter for the site is assigned a geometric that is $\ge 20L$ (called a stubborn choice) then with high probability it will not change back to 0 before $7NL$ updates have been done. Thus if we can show that there are at least $1.1pN$ vertices in $S$ that flip to 1 by time $7NL$ and are associated with stubborn choices we will contradict Lemma \ref{densest} below, which shows that with high probability the fraction of vertices in state 1 will be $\in(p-\ep,p+\ep)$ up to that time. The last result holds because only voter events change the number of 1's and we expect $7\nu N$ of them by time $7NL$. \eopt 

\medskip
The last two results are proved in Section \ref{sec:BSth}. Based on the fact that the bounds do not depend on $L$,
and on our analysis of approximate models below, which remove the dependence on $L$ by letting $L\to\infty$,
we conjecture that the critical value $\nu_c$ does not depend upon $L$, or to be precise

\begin{conj}
If we let $L=N^a$ and $N\to\infty$ then the limiting critical value does not depend upon $a$.
\end{conj} 

In support of this conjecture, Figure \ref{fig:2500vs25} gives a simulation of the system with $N=2500$, $L=25$, $\nu=2.5$ and
shows that the $(N_1(t),N_{10}(t))$ is well approximated by the quadratic equation $1.9059x(1-x)-0.11633$.
In Figure \ref{fig:2500vs50} we saw that for the system with $N=2500$, $L=50$, $\nu=2.5$ the curve 
 $(N_1(t),N_{10}(t))$ is well approximated by the quadratic equation $1.9328x(1-x)-0.11464$.

\begin{figure}[h] 
  \centering
  \includegraphics[width=5.67in,height=2.7in,keepaspectratio]{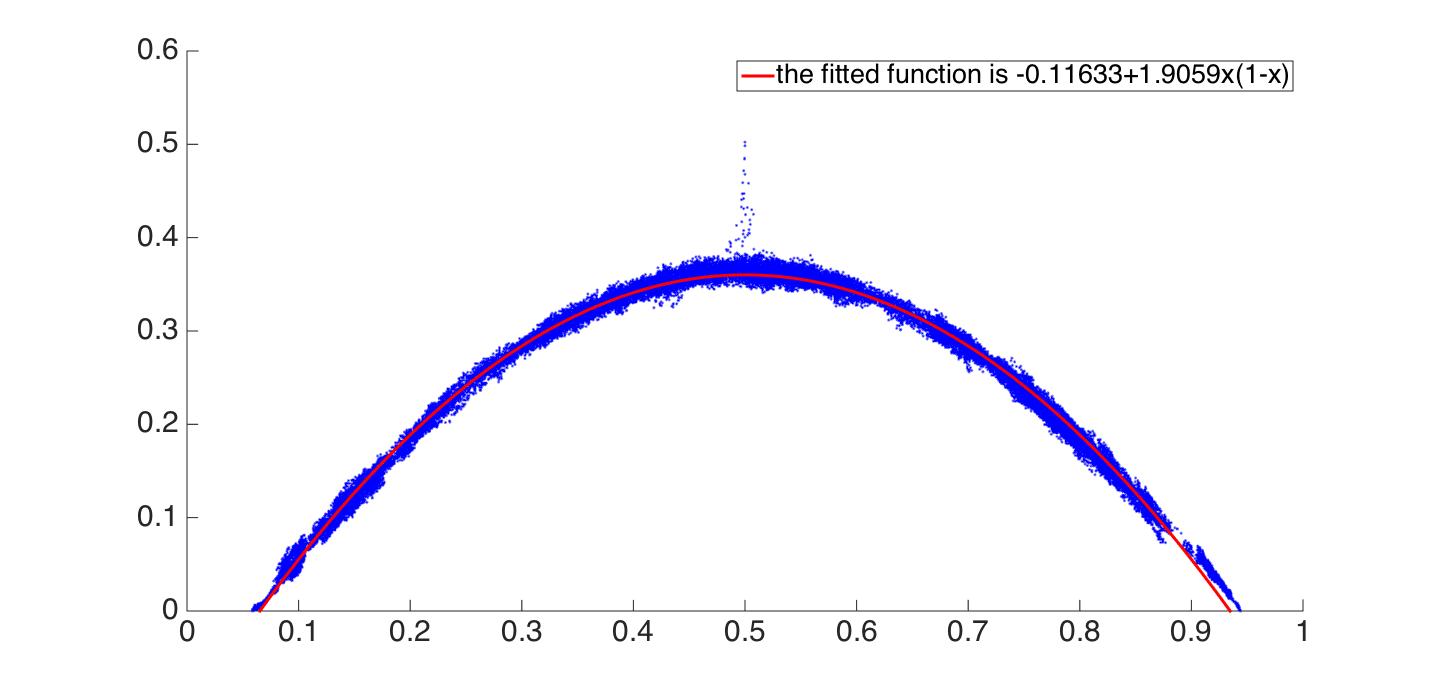}
  \caption{Simulation with $N=2500$, $L=25$, $\nu = 2.5$}
  \label{fig:2500vs25}
\end{figure}

\clearp

\section{Exact Equations and Pair Approximation}\label{sec:EEPA}

The results of Basu and Sly \cite{BaSly} described in the previous section establish the existence of a phase transition, but do not give very much information
about it. Our first step in obtaining more detailed (but approximate) results for the evolving voter model is to write down evolution equations for ``finite dimensional distributions.'' Define
\begin{align*}
N_{i} & = \sum_{x} 1_{\{\xi(x)=i\}}, \\
N_{ij} & = \sum_{x, y \sim x} 1_{\{\xi(x)=i, \xi(y)=j\}}, \\
N_{ijk} & = \sum_{x, y \sim x, z \sim y, z \neq x} 1_{\{\xi(x)=i, \xi(y)=j,\xi(z)=k \} },
\end{align*}
where $y \sim x$ means $y$ is a neighbor of $x$, and $\xi(x)$ denote the opinion of vertex $x$.  More abstractly, in the terminology of the theory of the convergence 
of random graphs $N_{ijk}$ is the number of homomorphisms of the small labeled graph drawn below to the one on $N$ vertices.

\begin{center}
\begin{picture}(180,50)
\put(30,20){\line(1,0){120}}
\put(28,17){$\bullet$}
\put(30,28){$i$}
\put(88,17){$\bullet$}
\put(90,28){$j$}
\put(148,17){$\bullet$}
\put(150,28){$k$}
\end{picture}
\end{center}

\noindent
Note that $N_{11}$ counts each $1-1$ twice, once for each orientation. Similarly $N_{101}$ counts each $1-0-1$ twice.
It is natural to think of these as finite-dimensional distributions but they are not. If we let $d(x)$ be the degree of vertex $x$, Then
\begin{align}
\sum_{i,j} N_{i,j} & = \sum_x d(x) = NL 
\label{sumij}\\
\sum_{i,j,k} N_{i,j,k} & = \sum_x d(x)(d(x)-1) \label{sumijk}
\end{align}
The first quantity is constant in time, but the second one is not.

Let $p = N_1/N$ be the initial fraction of vertices in state 1. Our first observation, which is implicit in the
arguments given in the previous section, is that in determining whether rapid disconnection
occurs we can suppose $p$ is constant \xyz{over the time scale if interest}. To do this we note that rewiring events do not change the number of 1's. 
The number of oriented edges is $NL$. Thus the number of 1's is can increase by 1 with rate at most $(NL/2)(\nu/L)$, 
and decrease by 1 with rate at most $(NL/2)(\nu/L)$. $N_1(t)/N$ is a martingale, so 
$$
E(N_1(t)/N - p)^2 \le \frac{\nu  t}{N},
$$ 
and therefore the fraction of 1's will not change significantly until times of $O(N)$.

Using reasoning from \cite{evo8} it is easy to show (see Section \ref{sec:dereq}) that
\begin{align}
\frac{dN_{10}}{dt} & = - N_{10} + \frac{\nu}{L} [N_{100} - N_{010} + N_{110} - N_{101}] 
\label{N10eq}\\ 
\frac{1}{2}\frac{dN_{11}}{dt} & = p N_{10} + \frac{\nu}{L} [N_{101} - N_{011}] 
\label{N11eq}\\
\frac{1}{2}\frac{dN_{00}}{dt} & = (1-p) N_{10} + \frac{\nu}{L} [N_{010} - N_{100}]
\label{N00eq}
\end{align}
Note that $N_{ij} = O(NL)$ while $N_{ijk} = O(NL^2)$ so the terms on the right-hand side of \eqref{N10eq}-\eqref{N00eq} are of the same order of magnitude.
In writing these equations we have omitted terms of the form $(\nu/L)N_{ij}$ since there are $O(N)$. Note that 
\eqref{sumij} implies  three equations sum to 0.

\subsection{Pair approximation (PA)}\label{sec:PA}

As is often the case in interacting particle systems, the derivatives of probabilities concerning two sites
involve three sites and if one writes differential equations for  probabilities concerning three sites
one gets expressions involving four sites. One way to deal with this problem is to use the {\it pair approximation}
to express three site probabilities in terms of the density of 1's and two site probabilities.
In the sparse graph case considered in \cite{evo8} this was an approximation that did not give a very good
answer, see Figure 9 there. 

When we began this research, here we thought (or hoped) that when the degrees are large, 
the pair approximation would give the right answer. Intuitively, if $y$ is a neighbor of $x$ then since $x$ is one 
of $O(L)$ neighbors of $y$ the state of $x$ has very
little influence on the state of $y$ and even less on the states of the neighbors of $y$. 

To do the pair approximation, 
let $J_i$ and $K_i$ be the average number of 1 neighbors and 0 neighbors of a vertex in state $i$.
By definition 
$$
N_1J_1=N_{11}, \qquad N_0K_0=N_{00}, \qquad N_1K_1=N_0J_0=N_{10}.
$$ 
The pair approximation in this context states that if $j_0(y)$ is the number of neighbors of 
a vertex $y$ in state 1 then
when we average over the neighbors of $x$, having opinion $0$, we get the mean $J_0$. That is, 
$$
N_{101} = \sum_{x: \xi(x)=1} \sum_{y : \xi(y) = 0} j_0(y) = N_{10}J_0.
$$

Applying similar reasoning for the other $N_{ijk}$'s we have
\begin{align}
\frac{1}{2}\frac{dN_{11}}{dt} & \approx p N_{10} + \frac{\nu}{L} [N_{10}J_0 - N_{01}J_1], 
\label{paN11}\\
\frac{1}{2}\frac{dN_{00}}{dt} & \approx (1-p) N_{10} + \frac{\nu}{L} [N_{01}K_1 - N_{10} K_0].
\label{paN00}
\end{align}
Analyzing these equations in Section \ref{sec:PT1} gives the following predictions about the means in equilibrium
\begin{align}
J_0^*  = L \left( 1 - \frac{p^2 +(1-p)^2}{\nu} \right) p,  & \qquad J_1^* = J_0^* + \frac{Lp}{\nu},
\label{paj}\\
K_1^*  = L \left( 1 - \frac{p^2 +(1-p)^2}{\nu} \right) (1-p),  & \qquad K_0^* = K_1^* + \frac{L(1-p)}{\nu}.
\label{pak}
\end{align}
 
In equilibrium we must have $J_0^*, K_1^* \ge 0$. This leads to the following:

\mn
{\bf Guess.} {\it In the rewire-to-random model, rapid disconnection occurs for $\nu< \nu_c(p) = p^2 + (1-p)^2$,
and prolonged persistence for $\nu > \nu_c(p)$.}

\mn
Unfortunately simulation shows that the second conclusion is not correct. If we let $N=1600$ and $L=40$ then this guess predicts that if the starting frequency of 1's is $=1/2$ then the phase transition occurs at $\nu=1/2$, while simulation
shows that rapid disconnection occurs for $\nu = 0.8$. To see the flaw in the intuition used earlier note that
$$
N_{101} = \sum_{ x : \xi(x) = 0 } j_0(x) \cdot (j_0(x)-1) \approx N_1 E [j_0(x)^2].
$$
Since $E [j_0(x)^2] > [E j_0(x)]^2 = J_0^2$ unless the distribution of $x$ is degenerate, we have
$$
N_{101} \ge J_0 N_1 J_0 = N_{10} J_0.
$$

\begin{conj}
The critical value from the pair approximation is a lower bound on the true value.
\end{conj}

\mn
If one can show that the pairs $(J_0,K_0)$ and $(J_1,K_1)$ are each negatively correlated this would follow.
This is far from obvious since $N_0$ and $N_1$ are random.

\clearp

\section{Approximate Master Equation (AME)}\label{sec:AME}

Most of the work in this paper is devoted to studying an improvement of the pair approximation that was also used in \cite{evo8}. 
For more on the PA and the AME and their use in studying dynamics on networks, see \cite{JPG1,JPG2}. The AME (i) uses ratios such as $N_{101}/N_{10}$
as parameters rather than approximating them by $N_{01}/N_0$ and (ii) tracks not only the means but the 
joint distribution of the state of a site and the number of neighbors 
with states 1 and 0. We visualize our system as
$N$ particles, one for each vertex, moving in two planes. A point at $(i,j,k)$ means that the state of the vertex is $i$,
there are $j$ neighbors in state 1, and $k$ in state 0. Voting events at the focal vertex $x$
cause jumping from $(1,j,k) \to (0,j,k)$ at rate $\nu k/L$ and from
$(0,j,k) \to (1,j,k)$ at rate $\nu j/L$. For the rewire-to-random model the transitions within each plane are as follows: 

\begin{center}
\begin{picture}(300,160)
\put(80,80){\vector(1,0){50}}
\put(85,120){$\frac{N_{10}}{N}$}
\put(80,80){\vector(0,1){50}}
\put(110,90){$\frac{N_{10}}{N}$}
\put(80,80){\vector(0,-1){50}}
\put(75,15){$k$}
\put(80,80){\vector(1,-1){50}}
\put(115,10){$\frac{kN_1}{N} + \frac{\nu}{L} k \frac{N_{101}}{N_{10}}$}
\put(80,80){\vector(-1,1){50}}
\put(10,135){$\frac{\nu}{L} j \frac{N_{110}}{N_{11}} $}
\put(20,30){plane 1}
\put(230,80){\vector(1,0){50}}
\put(235,120){$\frac{N_{10}}{N}$}
\put(230,80){\vector(0,1){50}}
\put(260,90){$\frac{N_{10}}{N}$}
\put(230,80){\vector(-1,0){50}}
\put(170,75){$j$}
\put(230,80){\vector(1,-1){50}}
\put(275,10){$\frac{\nu}{L} k \frac{N_{001}}{N_{00}} $}
\put(230,80){\vector(-1,1){50}}
\put(155,135){$\frac{jN_0}{N} + \frac{\nu}{L} j \frac{N_{010}}{N_{01}}$}
\put(190,30){plane 0}
\end{picture}
\end{center}

Here the rates on horizontal and vertical edges which come from rewiring are exact. On the diagonal arrows
$kN_1/N$ and $jN_0/N$ are exact but the others come from e.g., using $N_{ijk}/N_{ij}$ to compute the expected
number of neighbors of $z$ in state $k$ when $x$ is in state $i$ and $y$ is in state $j$. It is important to 
note that while the number of edges is conserved is in the original model, that is not true for our approximation.

On the time scale of our calculation $N_1=Np$ stays constant,
so the dynamics in plane 1 can be expressed as
\begin{align}
\frac{dj_1}{dt} & =  \frac{N_{10}}{N} \hphantom{-\,k_1\,} + p k_1 + \frac{\nu}{L}\cdot \frac{N_{101}}{N_{10}}k_1 
- \frac{\nu}{L} \cdot \frac{N_{110}}{N_{11}}j_1,\label{eqj1}\\
\frac{dk_1}{dt} &  = \frac{N_{10}}{N} -k_1 - p k_1 - \frac{\nu}{L}\cdot \frac{N_{101}}{N_{10}} k_1 
+ \frac{\nu}{L} \cdot \frac{N_{110}}{N_{11}}j_1.\label{eqk1}
\end{align}
Writing $q=1-p$ the plane 0 dynamics are
\begin{align}
\frac{dj_0}{dt} & =  \frac{N_{10}}{N}  - j_0 - q j_0 - \frac{\nu}{L}\cdot \frac{N_{010}}{N_{01}}j_0 
+ \frac{\nu}{L} \cdot \frac{N_{001}}{N_{00}}k_0,\label{eqj0}\\
\frac{dk_0}{dt} &  = \frac{N_{10}}{N} \hphantom{-\,k_1\,}+ q j_0 + \frac{\nu}{L}\cdot \frac{N_{010}}{N_{00}} j_0 
- \frac{\nu}{L} \cdot \frac{N_{001}}{N_{00}}k_0.\label{eqk0}
\end{align}
To go from the first set to the second exchange $j \leftrightarrow k$, $0 \leftrightarrow 1$, and change $p$ to $q$.

To study this system, we will introduce
\beq
\alpha = \frac{N_{101}}{N_{10}}, \qquad \beta = \frac{N_{110}}{N_{11}}, \qquad \eta = \frac{N_{10}}{N}  \qquad
\delta =  \frac{N_{010}}{N_{01}}, \qquad \ep = \frac{N_{001}}{N_{00}}.
\label{grdef}
\eeq
and analyze the general system
\begin{align}
\frac{dj_1}{dt} & =  \eta \hphantom{-\,k_1\,} + p k_1 + \frac{\nu}{L}\alpha k_1 
- \frac{\nu}{L} \beta j_1,\label{grj1}\\
\frac{dk_1}{dt} &  = \eta -k_1 - p k_1 - \frac{\nu}{L}\alpha k_1 
+ \frac{\nu}{L} \beta j_1,\label{grk1}\\
\frac{dj_0}{dt} & =  \eta  - j_0 - q j_0 - \frac{\nu}{L}\delta j_0 
+ \frac{\nu}{L} \ep k_0,\label{grj0}\\
\frac{dk_0}{dt} &  = \eta \hphantom{-\,k_1\,}+ q j_0 + \frac{\nu}{L}\delta j_0 
- \frac{\nu}{L} \ep k_0.\label{grk0}
\end{align}
Once this is done we will look for self-consistent parameters, i.e., values of the Greek letters so that \eqref{grdef} holds
in equilibrium.

Suppose for the moment that there are no jumps between planes.
To find the equilibrium in plane 1, we add \eqref{grj1} and \eqref{grk1} to get 
$$
\frac{d(j_1+k_1)}{dt} = 2\eta -k_1.
$$
so in equilibrium $k_1^* = 2\eta$. Using this in \eqref{grj1} we have
$$
\frac{\nu}{L} \cdot \beta j_1^* = \eta \left( 1 + 2p + 2 \alpha\frac{\nu}{L} \right).
$$
A similar calculation shows that $j_0^* = 2\eta$ and
$$
\frac{\nu}{L} \cdot \ep k_0^* = \eta \left( 1 + 2q + 2 \delta\frac{\nu}{L} \right).
$$
This equilibrium $(j_1^*,k_1^*)$ is globally attracting because, as we show in Section \ref{sec:ODE}, 
the linear differential equations in \eqref{grj1} and \eqref{grk1} have solution
$$
\begin{pmatrix}
j_1(t)- j_1^*(t)\\
k_1(t)- k_1^*(t)
\end{pmatrix}= \exp[ A t] \begin{pmatrix} 
j_1(0)- j_1^*(0)\\
k_1(0)-k_1^*(0)
\end{pmatrix},
$$
where $A$ is a matrix with two negative real eigenvalues. See Section \ref{sec:ODE} for details.

To analyze our two plane system, we take advantage of results of Lawley, Mattingly, and Reed \cite{LMR}. To put our system into their setting,
we assume that the values of $N_1/N$, $N_0/N$ and $N_{10}/NL$ are fixed. When this holds the individual particles move independently. If we let
$N\to\infty$ scale space by $L$, and suppose
\beq
\frac{\alpha}{L} \to \bar\alpha,   \qquad \frac{\beta}{L} \to \bar \beta,  \qquad  \frac{\eta}{L} \to \bar \eta, 
\qquad \frac{\delta}{L} \to \bar \delta,   \qquad \frac{\ep}{L} \to \bar\ep. 
\label{bargrdef}
\eeq
then in the limit we get a one particle system that moves according to the following  differential equations in plane 1
\begin{align*}
\frac{dx_1}{dt} & =  \bar\eta \hphantom{-x_1\,} + p y_1 + \nu \bar\alpha y_1 
- \nu  \bar \beta x_1,\\
\frac{dy_1}{dt} &  = \bar\eta -y_1 - p y_1 - \nu \bar\alpha y_1 
+ \nu \bar \beta x_1.
\end{align*}
and in plane 0 according to
\begin{align*}
\frac{dx_0}{dt} &  = \bar\eta - x_0 - q x_0 - \nu\bar\delta x_0 
+ \nu \bar\ep y_0, \\
\frac{dy_0}{dt} & =  \bar\eta \hphantom{-x_0\,} + q x_0 + \beta \bar\delta x_0 
- \beta \bar\ep y_0.
\end{align*}
The particle jumps from plane 1 to plane 0 at rate $\nu y_1$ and from plane 0 to plane 1 at rate $\nu x_0$.

If there are no jumps between planes then the system in plane 1 has a fixed point with $y_1^* = 2\bar\eta$ and
$$
\nu\bar\beta x_1^* = \bar\eta \left( 1 + 2p + 2\nu\bar\alpha \right),
$$
A similar calculation shows that $x_0^* = 2\eta$ and
$$
\nu \bar\ep y_0^* = \bar\eta \left( 1 + 2q + 2 \nu \bar \delta  \right).
$$
By almost exactly the same reasoning used on the previous system, the fixed points in each plane are globally attracting. Building in this, in Section \ref{sec:Ctoeq} we prove the following theorem.

\begin{theorem} \label{NLMtoeq}
Fix $\nu>0$, $p\in (0,1)$ and let $q=1-p$. For any $\bar\alpha, \bar\beta, \bar\gamma, \bar\ep, \bar\eta >0$
The two plane system has a unique stationary distribution that is the limit starting from
any initial configuration.
\end{theorem}

The proof, which we learned from \cite{LMR}, 
 is based on a variant of a trick used in random matrices (and other {subjects}). To prove that the product
$A_1 \cdots A_n$ converges in distribution, we show that $A_{-n} \cdot A_{-n+1} \cdots A_{-1}$ converges almost surely.
To apply this trick we first study the embedded discrete time that tracks the locations of the particle 
when the process changes planes. Using the fact that the linear ODEs in each plane are contractions the almost
sure convergence of the backwards version is easy. To get from this to the convergence of our continuous time
process we use a little (Markov)  renewal theory. See Section \ref{sec:Ctoeq} for details.

To get a feel for what the stationary distribution looks like, we turn to simulation. Figure \ref{fig:oneps} 
shows a simulation with $\bar\alpha = 0.3625$, $\bar\beta=0.3074$, $\bar\eta=0.0833$. These values correspond to the system with
$\nu=2$, see the first line in Table \ref{table:simvsth} in Section 5. 

\begin{figure}[h] 
  \centering
  \includegraphics[width=5.67in,height=2.7in,keepaspectratio]{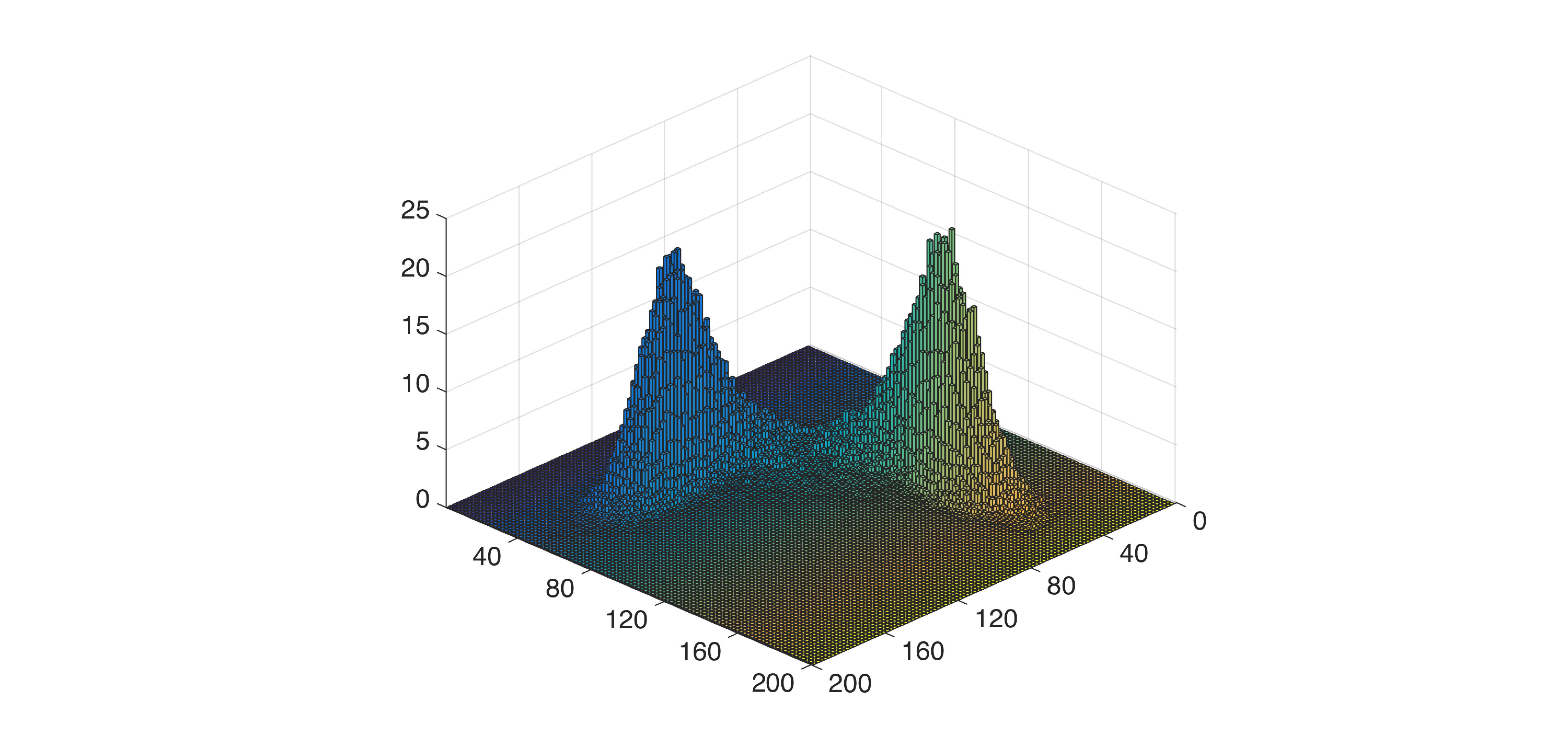}
  \caption{Picture of the stationary distribution for the one particle chain. The curves for 
planes 1 and 0 are plotted on the same graph.}
  \label{fig:oneps}
\end{figure}

\clearp 

\section{Generating Functions, PDE}\label{sec:GFPDE}

Theorem \ref{NLMtoeq} only asserts the existence of a unique stationary distribution in the two plane system, and it does not provide any further information. In this section we begin the process of identifying the limit distribution in Theorem \ref{NLMtoeq}. 
The calculations here are inspired by work of Silk et al.~\cite{Silk}, so we change our dynamics slightly to match theirs.
In our revised dynamics each oriented edge $(x,y)$ is chosen 
at rate 1 then $x$ imitates $y$ with probability $\nu/L$, and $x$ rewires to a new neighbor with probability $1 - {\nu/L}$.  
Let $A_{k,\ell}(t)$ ($B_{k,\ell}(t)$) be fraction of vertices in state 1 (0) with $k$ neighbors in state 1, and $\ell$ neighbors in
state 0 at time $t$. For ease of writing, we suppress the dependence on $t$, and continue to write $A_{k,\ell}$, and $B_{k,\ell}$ instead. Note that with this definition $\sum_{k,\ell} A_{k.\ell} = p$ the fraction of vertices in state 1 (recall at the time scale in which we are interested the fraction of $1$ does not change).

Let $Q(t,x,y) = \sum_{k,\ell} A_{k,\ell} x^k y^\ell$ and $R(t,x,y) = \sum_{k,\ell} B_{k,\ell} x^k y^\ell$.
Writing $Q_x$, $Q_y$, etc for partial derivatives, one can with some patience (see Section \ref{sec:derPDE}) arrive at
\begin{align}
Q_t & =  \frac{\nu}{L}\beta (y-x) Q_x  +  \frac{\nu}{L}xR_x 
\nonumber\\
& + \left( \left[ \frac{\nu}{L} + \alpha \frac{\nu}{L} + p ( 1 - \nu/L) \right] (x-y)  + (1-y) - \frac{\nu}{L} \right) Q_y 
\label{Qeqr}\\
& + ( 1 - \nu/L) \eta (x-1) Q + ( 1 - \nu/L) \eta (y-1) Q. 
\nonumber 
\end{align}
where we have set $\eta = N_{10}/N$.

Writing $q$ for $1-p$, similar reasoning gives
\begin{align}
R_t & =  \left( \left[ \frac{\nu}{L} + \delta \frac{\nu}{L} + q ( 1 - \nu/L) \right] (y-x) + (1-x) -\frac{\nu}{L} \right) R_x 
\label{Reqr}\\
& +  \frac{\nu}{L} yQ_y  + \frac{\nu}{L}\ep (x-y) R_y  + ( 1 - \nu/L) \eta (x-1) R + (1-\nu/L) \eta (y-1) R. 
\nonumber
\end{align}
To go from the $Q$ equation to the $R$ equation, interchange the roles of $x$ and $y$ and change the constants $\alpha \to \delta$,
$\beta \to \ep$, $p \to q$.

Silk et al.~\cite{Silk} considered ``rewire to same.'' {If in their notation we take $p=1-\bar p$ and $\bar p =\nu/L$ then their equation (13) becomes} 
\begin{align*}
Q_t & =  \frac{\nu}{L}\beta (y-x) Q_x  +  \frac{\nu}{L}xR_x \\
& + \left( \left[ 1 + \alpha \frac{\nu}{L} \right] (x-y)  + (1-y) - \frac{\nu}{L} \right) Q_y 
 + ( 1 - \nu/L) \gamma (x-1) Q, \\
R_t & =  \left( \left[ 1 + \delta \frac{\nu}{L} \right] (y-x) + (1-x) -\frac{\nu}{L} \right) R_x \\
& +  \frac{\nu}{L} yQ_y  + \frac{\nu}{L}\ep (x-y) R_y  + ( 1 - \nu/L) \zeta (x-1) R.  
\end{align*}
where $\gamma = Q_y(1,1)/Q(1,1) = N_{10}/N_1$ and $\zeta = N_{10}/N_0$ replace $\eta = N_{10}/N$.
Inside the square brackets in \eqref{Qeqr} and \eqref{Reqr} $N_1/N$ and $N_0/N$ are replaced by 1 
and the last two terms collapse to one in rewire to same, because in that case rewiring cannot cause a 0 to become a neighbor of a 1.

\subsection{Limit as $L\to\infty$} \label{sec:Limit as Ltoinfty}

Let $J_i$ ($K_i$) be the number of neighbors in state 1 (0) for a randomly chosen vertex in state $i$. Now we express the notations $\alpha,\beta, \eta$ etc in terms of functions of $J_i$ and $K_i$. To begin, we note that  $K_1=N_{10}/N_1$, and also $\eta = N_{10}/N$.
\begin{align*}
\alpha & = \frac{R_{xx}(1,1)}{R_x(1,1)} = \frac{E(J_0(J_0-1))}{EJ_0}, \\
\beta  & = \frac{Q_{xy}(1,1)}{Q_x(1,1)} = \frac{EJ_1K_1}{EJ_1}, \qquad
\eta = p EK_1.
\end{align*}
If $p=N_1/N$ is the fraction of vertices in state 1 then as $L \to \infty$
\begin{align*}
& Q(1+a/L,1+b/L) \to U(a,b) = p E \exp( a\bar J_1 +  b\bar K_1 ), \\
& R(1+a/L,1+b/L) \to V(a,b) = (1-p) E \exp( a\bar J_0 +  b\bar K_0 ). 
\end{align*}
where $\bar J_i$ and $\bar K_i$ are the limits in distribution of $J_i/L$ and $K_i/L$.
To derive the partial differential equations that $U$ and $V$ satisfy we note that
$$
U_a(a,b) = p E[\bar J_1 \exp(  a \bar J_1 + b \bar K_1 ) ],
$$
while
\begin{align*}
\frac{1}{L} Q_x (1+a/L,1+b/L) & = \sum_{j,k} \frac{j}{L} (1 + a/L)^{j-1} (1+b/L)^k A_{j,k} \\
& \to  p E[ \bar J_1 \exp( a \bar J_1 + b \bar K_1 ) ]  = U_a(a,b).
\end{align*}
Plugging in $x=1+a/L$, and $y=1+b/L$ in \eqref{Qeqr} and using $p=N_1/N$ we have
\begin{align*}
0 & =  \frac{\nu}{L}\beta \left(\frac{b}{L}-\frac{a}{L} \right) Q_x \left(1+\frac{a}{L},1+\frac{b}{L}\right)  \\
& +  \frac{\nu}{L} (1+a/L) R_x\left(1+\frac{a}{L},1+\frac{b}{L}\right) \\
& + \left( \left[ \frac{\nu}{L} + \alpha \frac{\nu}{L} + p ( 1 - \nu/L) \right]
\left(\frac{a}{L}-\frac{b}{L} \right) 
- \frac{b}{L} -\frac{\nu}{L} \right) Q_y\left(1+\frac{a}{L},1+\frac{b}{L}\right)\\
& + ( 1- \nu/L) \eta \left(  \frac{a}{L} \right) Q\left(1+\frac{a}{L},1+\frac{b}{L}\right) 
 + ( 1- \nu/L) \eta \left(  \frac{b}{L} \right) Q\left(1+\frac{a}{L},1+\frac{b}{L}\right).
\end{align*}
Using $\bar\alpha$, $\bar\beta$ and $\bar\eta$ as the limits of $\alpha/L$, $\beta/L$, $\eta/L$ we have
\beq
0 = \nu \bar\beta (b-a) U_a + \nu V_a + ( [p +\bar\alpha \nu] (a-b) - b - \nu) U_b + \bar\eta a U + \bar\eta b U.
\label{Ueqr}
\eeq 
Similarly,  
\beq
0 = \nu \bar\ep (a-b) V_b + \nu U_b + ( [p +\bar\delta \nu] (b-a) - a - \nu) V_a + \bar\eta a V + \bar\eta b V.
\label{Veqr}
\eeq 
This concludes the derivation of the limiting PDEs.


\section{Analysis for the symmetric case $p=1/2$}\label{sec:phalf}

In Section \ref{sec:GFPDE} we obtained limiting PDEs for the generating functions $Q$ and $R$ of $(A_{k,\ell})_{k,\ell}$, and $(B_{k,\ell})_{k,\ell}$, respectively. In this section our goal is to obtain solutions that satisfy the PDEs \eqref{Ueqr} and \eqref{Veqr}. To this end, we restrict ourselves to the symmetric case, i.e.~$p=1/2$. It will be evident from below that symmetry plays a crucial role in this computation.
In this symmetric case $Q(a,b) = R(b,a)$, so we  look for solutions of the form  
\beq
U(a,b) = \sum_{m,n} c_{m,n} a^m b^n, \qquad V(a,b) = \sum_{m,n} c_{n,m} a^m b^n.
\label{psol}
\eeq
Calculations given in Section \ref{sec:momeq} show that the coefficients satisfy
\begin{align}
0 & = \bar \eta c_{m-1,n} + \bar\eta c_{m,n-1} 
\nonumber\\
& +  \nu\bar\beta c_{m+1,n-1} (m+1) - (\nu\bar\beta m + (3/2 + \nu\bar\alpha) n) c_{m,n}   + (1/2+\nu\bar\alpha) c_{m-1,n+1}(n+1) 
\label{recr}
\\
& + \nu c_{n,m+1}(m+1) - \nu c_{m,n+1} (n+1). \nonumber
\end{align}
Notice that there are terms of order $m+n-1$, $m+n$ and $m+n+1$. Now we try to solve for the unknown coefficients $c_{m,n}$. Since $m! n!c_{m,n}= \partial a^m \partial b^n U(0,0)$, we therefore obtain equations involving different partial derivatives of $U$ evaluated at $(0,0)$. Since all the derivatives are evaluated at $(0,0)$, for convenience in writing, we suppress the argument $(0,0)$.

As we noted earlier, degree is not conserved in the approximate master equation. 
If the average degree in equilibrium is $L$ then when $p=1/2$ the average degrees
of vertices in states $i=1,0$ is $L$ by symmetry and we have
\beq
U_a + U_b = 1/2.
\label{E0}
\eeq
We say that the system is {\it conservative} in this case.

\mn
{\bf Zeroth order.} If we set $a=b=0$ in \eqref{Ueqr} then we get
$$
0 = \nu V_a - \nu U_b.
$$
This holds by symmetry, but is true in general since it says
$(1-p) E\bar J_0 = p E\bar K_1$ or $N_0 LE\bar J_0 = N_1 L E \bar K_1$,
which is true since each side is $N_{10}$.

\mn
{\bf First order.}  If we take $m=1,0$ and $n=1-m$ in \eqref{recr} then we find (see Section \ref{sec:momeq} for details)
that in general 
$$
\bar\eta = U_b,
$$ 
and we have
\beq
0 = -\nu\bar\beta U_a + (1+\nu\bar\alpha) U_b + \nu(U_{bb} - U_{ab}).
\label{E1gen}
\eeq

For self-consistent solutions we can use 
$$
\bar\alpha = \frac{U_{bb}}{U_b} \qquad \text{ and } \qquad \bar\beta = \frac{U_{ab}}{U_a} 
$$
to simplify this to 
$$
0  = - \nu U_{ab}   + U_b + \nu U_{bb}  + \nu U_{bb} - \nu U_{ab},
$$
or rearranging we have
\beq
U_b = 2\nu[U_{ab}-U_{bb}].
\label{E1r}
\eeq
which is a close relative of \eqref{N11eq}.

\mn
{\bf Second order.} Taking $m=2,1,0$ and $n=2-m$ in \eqref{recr} then we conclude (again see Section \ref{sec:momeq} for details)
\begin{align}
&(\nu/2)(U_{aab} -  U_{bbb})  = \bar \eta U_a + U_{ab}/2 + \nu (\bar\alpha U_{ab} - \bar\beta U_{aa} ),   
\nonumber\\
0 & = \bar \eta(U_a+ U_b) -3U_{ab}/2 + U_{bb}/2 + \nu[ \bar\alpha (U_{bb} - U_{ab}) ]  + \bar\beta (U_{aa} - U_{ab}) ], 
\label{E21r}\\
&(\nu/2)(U_{bbb} - U_{aab})  = \bar \eta U_b  - 3U_{bb}/2 + \nu[\bar\beta U_{ab} - \bar\alpha U_{bb}].
\nonumber
\end{align}
adding the equations we get 
\beq
 U_{ab} + U_{bb} = U_b.  \label{E20r}
\eeq

At this point we have four equations for our five unknowns $U_a$, $U_b$, $U_{aa}$, $U_{ab}$ and $U_{bb}$, but this still allows us
to compute all of them in terms of $U_b$. In the self-consistent case using \eqref{E1r} and \eqref{E20r} we get
\begin{align}
U_{ab} = \frac{1}{2} \left( 1 + \frac{1}{2\nu} \right) U_b, 
& \qquad \bar\beta = \frac{U_{ab}}{U_a} = \left( 1 + \frac{1}{2\nu} \right) \frac{U_{b}}{1-2U_b}, 
\label{E2abr} \\
U_{bb} = \frac{1}{2} \left( 1 - \frac{1}{2\nu} \right) U_b, 
& \qquad \bar\alpha = \frac{U_{bb}}{U_b} = \frac{1}{2} \left( 1 - \frac{1}{2\nu} \right).
\label{E2bbr}
\end{align}

To compute $U_{aa}$ now, we note that $\bar\alpha=U_{bb}/U_b$ and $U_{ab} - U_{bb} = U_b/2\nu$ so $\nu\bar\alpha(U_{bb}-U_{ab})=-U_{bb}/2$
and \eqref{E21r} simplifies to
$$
0  = \bar\eta/2 - 3U_{ab}/2 +  \nu [ \bar\beta (U_{aa} - U_{ab}) ].
$$
Rearranging
$$
- \bar\eta = 2\nu\bar\beta U_{aa} - (2 \nu\bar\beta + 3) U_{ab},
$$
so we have 
\beq 
U_{aa} = \left( 1 + \frac{3}{2\nu\bar\beta} \right) U_{ab} - \frac{\bar\eta}{2\nu\bar\beta}
\label{E23r}
\eeq

To get a sense of the accuracy of the approximate master equation we simulate the system with $N=1600$, $L=40$ {to find $U_b$
then use the equations \eqref{E2abr}, \eqref{E2bbr}, and \eqref{E23r}.
The predicted values of $U_{ab}$ and $U_{bb}$ given in Table \ref{table:simvsth}} agree well with those from simulation, having errors that are mostly
about 1\%. However the predictions for $U_{aa}$ have errors of about 10\%, showing that 1's are more clustered
than the approximate master equation predicts.
 

\begin{table}[htbp]
\begin{center}
\begin{tabular}{l|c|cc|cc|cc}
$\nu$ & $U_b$ sim & $U_{ab}$ sim & \eqref{E2abr} & $U_{bb}$ sim & \eqref{E2bbr} & $U_{aa}$ sim & \eqref{E23r} \\
\hline
2 & 0.1666 & 0.1025 & 0.1041 & 0.0604 & 0.0625 & 0.2336 & 0.2208 \\
1.6 & 0.1371 & 0.0907 & 0.0900 & 0.0466 & 0.0471 & 0.2859 & 0.2574 \\
1.44 & 0.1216 & 0.0827 & 0.0819 & 0.0394 & 0.0397 & 0.3115 & 0.2810 \\
1.32 & 0.1094 & 0.0757 & 0.0754 & 0.0343 & 0.0340 & 0.3310 & 0.3047 \\
1.2 & 0.0896 & 0.0641 & 0.0635 & 0.0264 & 0.0261 & 0.3735 & 0.3351 \\
1 & 0.0454 & 0.0339 & 0.0341 & 0.0132 & 0.0113 & 0.4690 & 0.4129
\end{tabular} 
\caption{Simulation of evolving voter model compared with computations for the approximate mater equation.}
\label{table:simvsth}
\end{center} 
\end{table}

If we go to third order then we have four new equations, see \eqref{order3} but we have three new equations for four new unknowns
so we are falling further behind. Despite this fact, as Silk et al.~\cite{Silk} explain, it is possible in the symmetric case to compute generating function and find
``self-consistent solutions,'' i.e., those that have the property that if we set the values of $\bar\alpha$, $\bar\beta$, and $\bar\eta$
and then compute the values of $U_{bb}/U_b$, $U_{ab}/U_a$, and $U_b$ they agree with the specified parameters. To do this they note that
if one specifies the values of the $U_{aaab} \ldots U_{bbbb}$ then one can solve for the lower order $U$'s then one has a fourth order
approximation to the solution. If we do the $n$th order approximation, choose the $n$th order variables so that $U_a + U_b =1/2$, and let $n\to\infty$ then the limit exists. They take the eighth order approximation and then use symbolic computation to find self-consistent values 
of $\bar\alpha$, $\bar\beta$, and $\bar\eta$. See their paper for results for rewire to same. 

Since this is beyond our computer skills we leave this as an exercise for more capable readers. We would be more excited if this method
could be used to get results for the general case, however symmetry seems crucial to the computation. 

\clearp

\section{Derivation of the Equations} \label{sec:dereq}

This section provides the derivation of the equations \eqref{N10eq}-\eqref{N00eq}. To this end, we begin by writing the equations in the notation of \cite{evo8}, i.e., $\alpha$ is the rewiring rate not the quantity
$N_{101}/N_{10}$ introduced in the discussion of the approximate master equation. 

\mn
{\bf Rewire-to-random.} We fix $(x,y)$ and consider $(d/dt) 1_{\{\xi(x)=i, \xi(y)=j\}}$ for all possible $i,j$.
To do this we consider the various possibilities for the oriented edge $(u,v)$ and which the 
update occurs and whether the even is voting or rewriting.

\mn
I. {\bf Pairs destroyed by rewiring.}

\mn
\begin{tabular}{cccc}
$u=x$ & $v=y$  & rate & destroy \\
 1 & 0 & $\alpha N_{10}$ & 10 \\
 0 & 1 & $\alpha N_{10}$ & 01 \\
$u=y$ & $v=x$  & rate & destroy \\
1 & 0 &  $\alpha N_{10}$ & 01 \\
0 & 1 & $\alpha N_{10}$ & 10
\end{tabular}

\mn
II. {\bf Pairs created by rewiring.} 

\mn
\begin{tabular}{ccccc}
$u=x$ & $v$ & $v'=y$ & rate & create \\
 1 & 0 & 0 & $\alpha N_{10}(1-p)$ & 10 \\
1 & 0 & 1 & $\alpha N_{10}p$ & 11 \\
 0 & 1 & 0 & $\alpha N_{10}(1-p)$ & 00 \\
0 & 1 & 1 & $\alpha N_{10}p$ & 01 \\
$u=y$ & $v$ & $v'=x$ & rate & create \\
 1 & 0 & 0 & $\alpha N_{10}(1-p)$ & 01 \\
1 & 0 & 1 & $\alpha N_{10}p$ & 11 \\
 0 & 1 & 0 & $\alpha N_{10}(1-p)$ & 00 \\
0 & 1 & 1 & $\alpha N_{10}p$ & 10 
\end{tabular}

\mn
III. {\bf Internal voting on $(x,y)$.}

\begin{tabular}{lccc}
10 vote & rate &  destroy & create \\                       
$uv$ & $(1-\alpha) N_{10}$ & 10  & 11 \\
$vu$ & $(1-\alpha) N_{01}$ & 01  & 11 \\
01 vote & rate &  destroy & create \\                       
$uv$ & $(1-\alpha) N_{01}$ & 01  & 00 \\
$vu$ & $(1-\alpha) N_{10}$ & 10  & 00 
\end{tabular}

\mn
IV. {\bf External Voting}

\mn
\begin{tabular}{ccccc}
$u \sim x = v$, & y   & rate & create & destroy \\
10 & 0 & $(1-\alpha)N_{100}$ & 10 & 00 \\
10 & 1 & $(1-\alpha)N_{101}$ & 11 & 01 \\
01 & 0 & $(1-\alpha)N_{010}$ & 00 & 10 \\
01 & 1 & $(1-\alpha)N_{011}$ & 01 & 11 
\end{tabular}

\mn
\begin{tabular}{ccccc}
$u \sim y = v$ & x  & rate & create & destroy \\
10 & 0 &  $(1-\alpha)N_{001}$ & 01 & 00 \\
10 & 1 & $(1-\alpha)N_{101}$ & 11 & 10 \\
01 & 0 & $(1-\alpha)N_{010}$ & 00 & 01 \\
01 & 1 & $(1-\alpha)N_{110}$ & 10 & 11 
\end{tabular}

\mn
Adding up the rates from the tables gives the following equations. Note that $N_{11}$, $N_{00}$, $N_{111}$ and $N_{000}$ do not appear on the right-hand side. Noting that $N_{110}=N_{011}$ and $N_{100}=N_{001}$ we have
\begin{align}
\frac{dN_{10}}{dt} & = -\alpha N_{10} + (1-\alpha) [-2 N_{10} + N_{100} - N_{010} + N_{110} - N_{101}], 
\nonumber\\ 
\frac{1}{2}\frac{dN_{11}}{dt} & = \alpha p N_{10} + (1-\alpha) [N_{10} + N_{101} - N_{011}], 
\label{rrNeq}\\
\frac{1}{2}\frac{dN_{00}}{dt} & = \alpha(1-p) N_{10} + (1-\alpha) [N_{10} + N_{010} - N_{100}].
\nonumber
\end{align}

\mn
We have separated the rewiring terms I+II multiplied by $\alpha$ from the voting terms III+IV
multiplied by $1-\alpha$. Simplifying gives 
\begin{align*}
\frac{dN_{10}}{dt} & = -(2-\alpha) N_{10} + (1-\alpha) [N_{100} - N_{010} + N_{110} - N_{101}], \\ 
\frac{1}{2}\frac{dN_{11}}{dt} & = (1-\alpha(1-p)) N_{10} + (1-\alpha) [N_{101} - N_{011}], \\
\frac{1}{2}\frac{dN_{00}}{dt} & = (1-\alpha p) N_{10} + (1-\alpha) [N_{010} - N_{100}].
\end{align*}
Let $d_x$ be the degree of $x$ and $M = \sum_x d_x$ be the number of oriented edges. Note that 
$$
N_{11} + 2N_{10} + N_{00} = \sum_x d_x = M \sim N^2p
$$
so the sum of the three equations must be 0 (and it is).

When $1-\alpha = \nu/L$ we have
\begin{align}
\frac{dN_{10}}{dt} & = - N_{10} + \frac{\nu}{L} [N_{100} - N_{010} + N_{110} - N_{101}], 
\nonumber\\ 
\frac{1}{2}\frac{dN_{11}}{dt} & = p N_{10} + \frac{\nu}{L} [N_{101} - N_{011}], 
\label{rrNeq2}\\
\frac{1}{2}\frac{dN_{00}}{dt} & = (1-p) N_{10} + \frac{\nu}{L} [N_{010} - N_{100}].
\nonumber
\end{align}

\clearp

\section{Pair approximation} \label{sec:PT1}
In this section we derive \eqref{paj}-\eqref{pak}, namely the means of $1$ neighbor, and $0$ neighbor under the pair approximation. To this end, we recall \eqref{paN11} and \eqref{paN00}
\begin{align}
\frac{1}{2}\frac{dN_{11}}{dt} & = p N_{10} + \frac{\nu}{L} [N_{10}J_0 - N_{01}J_1], 
\label{paN11copy}\\
\frac{1}{2}\frac{dN_{00}}{dt} & = (1-p) N_{10} + \frac{\nu}{L} [N_{01}K_1 - N_{10} K_0].
\label{paN00copy}
\end{align}
In equilibrium we have
$$
p + \frac{\nu}{L} (J_0-J_1) = 0, \qquad (1-p) + \frac{\nu}{L} (K_1-K_0) = 0. 
$$
or rearranging 
\beq
J_1 - J_0 = \frac{Lp}{\nu}, \qquad K_0 - K_1 =  \frac{L(1-p)}{\nu}.
\label{e1}
\eeq
To have four equations we recall that 
\begin{align}
&p K_1=(1-p)J_0,
\label{e2}\\
&p(J_1+K_1) + (1-p)(J_0+K_0) = L.
\label{e3}
\end{align}

To simplify the equations we begin by noting that using \eqref{e3} with \eqref{e1} gives
\begin{center}
\begin{tabular}{ccccl}
$p J_1$ & $+p K_1$ & $+(1-p)J_0$ & $+(1-p)K_0$ & $=L$ \\
$(1-p)J_1$ & & $-(1-p) J_0$ & & $=Lp(1-p)/\nu$\\
& $-p K_1$ & & $+p K_0$ &  $=Lp(1-p)/\nu$
\end{tabular}
\end{center}
so we have
\beq
J_1 + K_0 = L + \frac{2Lp(1-p)}{\nu}.
\label{e4}
\eeq
Adding the equations in \eqref{e1}
$$
J_1 + K_0 - J_0 - K_1 = \frac{Lp}{\nu} + \frac{L(1-p)}{\nu}.
$$
Using \eqref{e4}
$$
L + \frac{Lp(1-p)}{\nu} + \frac{Lp(1-p)}{\nu} = J_0 + K_1 +\frac{Lp}{\nu} + \frac{L(1-p)}{\nu},
$$
so we have
$$
L - \frac{Lp^2}{\nu} - \frac{L(1-p)^2}{\nu}= J_0 + K_1 .
$$
Using \eqref{e2} now and noting $J_0 + (1-p)J_0/p = J_0/p$ we have that in equilibrium
\begin{align}
J_0^* & = L \left( 1 - \frac{p^2 +(1-p)^2}{\nu} \right) p, 
\label{j0*}\\
K_1^* & = L \left( 1 - \frac{p^2 +(1-p)^2}{\nu} \right) (1-p).
\label{k1*}
\end{align}
To finish up we note that from \eqref{e1}
\begin{align}
J_1^* = J_0^* + \frac{Lp}{\nu} & = L \left( 1 + \frac{1}{\nu} - \frac{p^2 +(1-p)^2}{\nu} \right) p,
\label{j1*} \\
K_0^* = K_1^* + \frac{L(1-p)}{\nu} & = L \left( 1 + \frac{1}{\nu} - \frac{p^2 +(1-p)^2}{\nu} \right) (1-p).
\label{k0*} 
\end{align}
For $J_0^*,K_1^*>0$ we must have $\nu > \nu_c(p) \equiv p^2 + (1-p)^2$. In this case we will also have
$J_1^*, K_0^*>0$. To begin to check \eqref{e3} we note that
\begin{align*}
J_1^* + K_1^* & = L \left( 1 + \frac{p}{\nu} - \frac{p^2 +(1-p)^2}{\nu} \right),  \\
J_0^* + K_0^* & = L \left( 1 + \frac{(1-p)}{\nu} - \frac{p^2+(1-p)^2}{\nu} \right).
\end{align*}
so we do have $p(J_1^*+K_1^*) + (1-p)(J_0^*+K_0^*) = L$.

\clearp

\section{Analysis of the single plane ODEs} \label{sec:ODE}
Recall that in Section \ref{sec:AME} we introduced approximate master equation, where we represent the states of a local neighborhood of a vertex by a triplet. Namely, in $(i,j,k)$  $i$ represents the state of the focal vertex, $j$ is the number of $1$ neighbors, and $k$ is the same for $0$ neighbors. This gives rise to a two plane system, and we claimed in Theorem \ref{NLMtoeq} that this two plane system has a unique stationary distribution. To prove Theorem \ref{NLMtoeq}, we first need to show that the sets of differential equations in the individual planes are globally attractive, which is done in this section. Building on this, we finish the proof of Theorem \ref{NLMtoeq} in the next section.

To this end, using \eqref{grj1}, \eqref{grk1} and subtracting
\begin{align*}
0 & =  \eta \hphantom{-\,k_1\,} + p k^*_1 + \frac{\nu}{L} \alpha k^*_1 
- \frac{\nu}{L} \beta j^*_1,\\
0 &  = \eta -k^*_1 - p k^*_1 - \frac{\nu}{L}\alpha k^*_1 
+ \frac{\nu}{L} \beta j^*_1,
\end{align*}
we conclude:
\begin{align*}
\frac{d(j_1-j_1^*)}{dt} & =  -  \frac{\nu}{L} \alpha(j_1 - j_1^*) 
+ \left(  p  + \frac{\nu}{L}\beta \right) (k_1-k_1^*), 
\\
\frac{d(k_1-k_1^*)}{dt} &  = \frac{\nu}{L} \alpha(j_1 - j_1^*) 
- \left(1+p+ \frac{\nu}{L}\beta \right) (k_1-k_1^*),
\end{align*}
Scaling by $L$ and letting $L\to\infty$
\begin{align*}
\frac{d(x_1-x_1^*)}{dt} & =  -  \nu\bar\alpha(x_1 - x_1^*) 
+ (  p  + \nu\bar\beta) (y_1-y_1^*), 
\\
\frac{d(y_1-y_1^*)}{dt} &  = \nu\bar\alpha(x_1 - x_1^*) 
- (1+p+ \nu\bar\beta) (y_1-y_1^*),
\end{align*}

Either equation can be written in matrix form as
$$
\begin{pmatrix}
u_1(t)\\
v_1(t)
\end{pmatrix}= A \begin{pmatrix} 
u_1(0)\\
v_1(0)
\end{pmatrix},
$$
where
$$
A= \begin{bmatrix}
- a & p+ b \\
a  & - ( 1+ p + b)
\end{bmatrix}.
$$
with $a,b>0$ so the solution is
$$
\begin{pmatrix}
u_1(t)\\
v_1(t)
\end{pmatrix}= \exp[ A t] \begin{pmatrix} 
u_1(0)\\
v_1(0)
\end{pmatrix},
$$
The trace of $A$ is $-(1 + p + a + b) < 0$ while the determinant is $a > 0$, so it is clear that both eigenvalues have negative real part.
To show that they are real we note that they satisfy
\beq
\det(\lambda I - A ) = \lambda^2 + (1 + p + a + b) \lambda + a = 0.
\label{eveq}
\eeq
Solving the quadratic equation we have
\beq
\lambda_i = \frac{-(1 + p + a + b) \pm \sqrt{ (1 + p + a + b)^2 - 4a }}{2}.
\label{lami}
\eeq
To show that the quantity under the square root is positive  we note that if $r = 1 + p + b$  then
$$
(1 + p + a + b)^2 - 4a  = (r+a)^2 - 4a = (r-a)^2 + 4ra - 4a > 0
$$
since $a > 0$ and $r> 1$.

\clearp

\section{AME: Convergence to Equilibrium} \label{sec:Ctoeq}
Building on the results of Section \ref{sec:ODE} we now finish the proof of Theorem \ref{NLMtoeq}. Recall that the set of equation for plane $i$ can be written in the following form:
$$
\begin{pmatrix}
u_i(t)\\
v_i(t)
\end{pmatrix}= \exp[ A_i t] \begin{pmatrix} 
u_i(0)\\
v_i(0)
\end{pmatrix},
$$
where both the eigenvalues of $A_i$ are real negative, and hence there exists a unique solution of the differential equation in plane $i$ starting from $z = (x_i,y_i)$, which is denoted hereafter by 
$\Phi^i_t(z)$.
It is clear that $t \to \Phi^i_t(z)$ are continuous. The matrix representations imply that if $t>0$
\beq
|\Phi^i_t(z) - \Phi^i_t(w)| \le K_i(t)|z-w|
\label{Phicon}
\eeq
with $K_i(t)<1$. Note that $K_i(s+t) \le K_i(s)K_i(t)$.

Let $T_i(z)$ be the time of the first jump to the other plane when the solution starts at $z$ in plane $i$. 
\cite{LMR} study the situation in which there are two differential equations and the $k$th pair of switching times between
them $(\tau_0^k, \tau_1^k)$ are independent and drawn from a distribution $\mu_0 \times \mu_1$.
In our situation the jump times depend on the starting point, but their method of proof extends easily. 
Let $F_{i,z}$ be the distribution of $T_i(z)$, let $U^1_k, U^0_k$ be i.i.d~uniform on $(0,1)$ and let
$\tau^k_i(z) = F^{-1}_{i,z}(U^k_i)$ which has the same distribution as $T_i(z)$. 
Define the compositions
\begin{align*}
G^k_\omega(z) = \Phi^1_{\tau^k_1(w)} (\Phi^0_{\tau^k_0(z)}(z)) \quad\hbox{where}\quad  w =  \Phi^0_{\tau^k_0(z)}(z), \\
H^k_\omega(z) = \Phi^0_{\tau^k_0(w)} (\Phi^1_{\tau^k_1(z)}(z)) \quad\hbox{where}\quad  w =  \Phi^1_{\tau^k_1(z)}(z).
\end{align*}

Define the forward maps $\phi^n$ and  $\gamma^n$ and the backwards maps $\phi^{-n}$ and  $\gamma^{-n}$ by
\begin{align*}
\phi^n_\omega(z) = G^n_\omega \circ \cdots \circ G^1_\omega(z), & \quad\hbox{and}\quad \gamma^n_\omega(z) = H^n_\omega \circ \cdots \circ H^1_\omega(z), \\
\phi^{-n}_\omega(z) = G^1_\omega \circ \cdots \circ G^n_\omega(z), & \quad\hbox{and}\quad \gamma^{-n}_\omega(z) = H^1_\omega \circ \cdots \circ H^n_\omega(z). 
\end{align*}
This is a well-known trick in the theory of iterated functions. The functions $\phi_n$ and $\phi_{-n}$ have the same distribution but $\phi_{-n}$s admit a almost sure limit:

\begin{lemma}
$Y_1(\omega) = \lim_{n\to\infty} \phi^{-n}_\omega(z)$ and $Y_0(\omega) = \lim_{n\to\infty} \gamma^{-n}_\omega(z)$
exist almost surely and are independent of $z$.
\end{lemma}

\noindent
This follows easily from the contraction property in \eqref{Phicon}, see the proof of Proposition 1 in \cite{LMR}.

$G^1_\omega(z)$ gives the location of the path on its first return to plane 0, so $Y_1$ gives the equilibrium distribution at that time.
Likewise, $H^1_\omega(z)$ gives the location of the path on its first return to plane 1, so $Y_0$ gives the equilibrium distribution at that time.
It follows easily from the existence of the limit that (this is Proposition 2 in \cite{LMR})

\begin{lemma}
$Y_0 =_d \Phi^0_{\tau_0(Y_1)}(Y_1)$ and $Y_1 =_d \Phi^1_{\tau_1(Y_0)}(Y_0)$.
\end{lemma}

The average time spent in plane 0 is $\nu_0 = E \tau_0(Y_1)$. The average time spent in plane 1 is $\nu_1 = E \tau_1(Y_0)$.
Once we show these are finite we can conclude that the long run fraction of time spent in plane 1 is $\nu_1/(\nu_1+\nu_0)$.
To compute the limiting behavior of the continuous time process the following picture is useful.

\begin{center}
\begin{picture}(285,70)
\put(30,30){\line(1,0){225}}
\put(28,27){$\bullet$}
\put(25,40){$Y^1_0$}
\put(50,50){$t^1_1$}
\put(17,15){$\sigma_0=0$}
\put(78,27){$\bullet$}
\put(75,40){$Y^1_1$}
\put(105,50){$t^1_0$}
\put(138,27){$\bullet$}
\put(135,40){$Y^2_0$}
\put(157,50){$t^2_1$}
\put(134,15){$\sigma_1$}
\put(178,27){$\bullet$}
\put(175,40){$Y^2_1$}
\put(205,50){$t^2_0$}
\put(238,27){$\bullet$}
\put(235,40){$Y^3_0$}
\put(234,15){$\sigma_2$}
\end{picture}
\end{center}

\noindent
The superscripts refer to time. 
In words, we start in plane 1 at location $Y^1_0 =_d Y_0$, we follow the ODE for  
time $t^1_1 = \tau_1(Y^1_0)$ when we jump to $Y^1_1 = \Phi^1_{\tau_1(Y^1_0)}(Y^1_0)$ in plane 0, etc.

Consider now the process $Z(t)$ that is constant on each time interval and equal to the value at the left endpoint of the interval.
$Z(t)$ is a Markov alternating renewal process. Recall that in an Markov renewal process if we jumped into state $X_k$ at time $T_k$ then the
next state $X_{k+1}$ and the waiting time $t_{k+1}$ until we jump to $X_{k+1}$ have a joint distribution $(X_{k+1},t_{k+1})$
that depends on $X_k$ but is otherwise independent of the past before time $T_k$.  
See e.g., Chapter 10 of Cinlar \cite{Cinlar}.
Our process is ``alternating'' because the joint distribution used alternates. 

Call the sojurn in plane 1 combined with the sojurn in plane 0, a cycle. Let $\rho_1$ be the distribution of $Y_0$ on plane 1, and
let $\rho_0$ be the distribution of $Y_1$ on plane 0. Let $\bar\rho_i$ be the measure with
$$
\frac{d\bar\rho_i}{d\rho_i}(z) = E\tau_i(z)
$$
$\bar\rho_i$ has total mass $\nu_i$. Let $\pi$ be the measure that has density $\pi_i = \bar\rho_i/(\nu_1+\nu_0)$ on plane $i$  
than applying the Ergodic Theorem to the sequence of cycles $\{ Z(\sigma_{k-1} +t), 0 \le t \le \sigma_{k} - \sigma_{k-1} \}$ 
shows that our alternating renewal process $Z(t)$ has stationary distribution $\pi$. To do this note that the cycles are
simply a Markov chain on a space of paths starting from its stationary distribution. 

Let $V(t)$ be the process that starts at $Y^1_0$ at time 0, is at $\Phi^1_t(Y^1_0)$ for $t < \tau_1(Y^1_0)$ when it jumps to
$Y^1_1$, etc. Let $L(t)$ be the time of the last jump before time $t$, let $g(t) = f_1(Z(t))1_{\{t-L(t)>x\} }$ where 
$f_1$ is a bounded function with $f_1=0$ on
plane 0. It follows from the applying the Ergodic Theorem to the sequence of cycles that
$$
\frac{1}{t} \int_0^t g(s) \, ds \to \frac{E( f_1(Y_0) (\tau_1(Y_0) - x)^+ )}{ \nu_1 + \nu_0 }.
$$ 
For the use of this idea in the simpler setting of renewal theory see Section 3.3.2 in \cite{EOSP}. 

The last result when supplemented by the analogous conclusion for a function $f_0$ that vanishes on plane 1 gives the limiting joint distribution
of $(Z(t),A(t))$ where $A(t) = t - L(t)$ is the age (time since the last jump) at time $t$. Differentiating with respect to $x$ we see that on plane $i$ that
joint distribution is given by
$$
\chi_i(z,a) = \frac{\rho_i(z) P( \tau_i(z) > a)}{\nu_1+\nu_0}.
$$
From this we can compute the limiting distribution of $V(t)$. If $f_1$ is as above
$$
\lim_{t\to\infty} Ef_1(V(t)) = \int_{[0,\infty)^2} dz \, \int_0^\infty da \, Ef_1(\Phi^1_a(z)).
$$

\clearp 

\section{Derivation of the PDE} \label{sec:derPDE}

In this section we provide the derivation of \eqref{Qeqr}, and \eqref{Reqr}. Writing $x$ for the focal vertex, $y$ for a neighbor, $z$ a neighbor of the neighbor $y$, and $w$ some other vertex in graph.
By patiently considering all the possible changes one finds:

\begin{align*}
\frac{dA_{k,\ell}}{dt}& = \frac{\nu}{L} [kB_{k,\ell} - \ell A_{k,\ell}] &\quad\hbox{vote $y\to x$}\\
& + \frac{\nu}{L} [(\ell+1)A_{k-1,\ell+1} - \ell A_{k,\ell} ]  &\quad\hbox{vote $x\to y$} \\ 
& + \frac{\nu}{L} \alpha[(\ell+1)A_{k-1,\ell+1} - \ell A_{k,\ell} ] &\quad\hbox{vote $z\to y$, $\xi(y)=0$}\\ 
& + \frac{\nu}{L} \beta [(k+1)A_{k+1,\ell-1} - k A_{k,\ell} ] &\quad\hbox{vote $z\to y$, $\xi(y)=1$} \\ 
& + ( 1 - \nu/L) [(\ell+1) A_{k,\ell+1} - \ell A_{k,\ell}] &\quad\hbox{$y$ rewires away from $x$} \\
& +  ( 1 - \nu/L) p[(\ell+1) A_{k-1,\ell+1} - \ell A_{k,\ell}] &\quad\hbox{$x$ rewires and connects to a 1}\\
& + ( 1 - \nu/L) \eta[A_{k-1,\ell} - A_{k,\ell}] &\quad\hbox{$w$ with $\xi(w)=1$ rewires to $x$} \\
& + ( 1 - \nu/L) \eta [A_{k,\ell-1} - A_{k,\ell}] &\quad\hbox{$w$ with $\xi(w)=0$ rewires to $x$}
\end{align*}
For the last two equations note that for each discordant edge, one of the orientations brings a 1, the other a 0.

Let $Q(t,x,y) = \sum_{k,\ell} A_{k,\ell} x^k y^\ell$ and $R(t,x,y) = \sum_{k,\ell} B_{k,\ell} x^k y^\ell$.
Writing $Q_x$, $Q_y$, etc for partial derivatives, the second terms in lines 1, 2, 3, 5, 6 are
$$
\sum_{k,\ell} \ell A_{k,\ell} x^k y^\ell = y \sum_{k,\ell} A_{k,\ell} x^k \ell y^{\ell-1} = y Q_y.
$$
Similarly the second term in line 4 and the first in line 1 are
$$
\sum_{k,\ell} k A_{k,\ell} x^k y^\ell = x Q_x \quad\hbox{and}\quad \sum_{k,\ell} k B_{k,\ell} x^k y^\ell = x R_x.
$$
The first terms in lines 2, 3, 6 are
$$
\sum_{k,\ell} (\ell+1) A_{k-1,\ell+1} x^k y^\ell = xQ_y.
$$
The first term in line 4 is
$$
\sum_{k,\ell} (k+1) A_{k+1,\ell-1} x^k y^\ell = yQ_x.
$$
The first term in line 5 is
$$
\sum_{k,\ell} (\ell+1) A_{k,\ell+1} x^k y^\ell = Q_y.
$$
The first terms in lines 7 and 8 are
$$
\sum_{k,\ell} A_{k-1,\ell} x^k y^\ell = x Q \quad\hbox{and}\quad \sum_{k,\ell} A_{k,\ell-1} x^k y^\ell = y Q.
$$

Combining the formulas for the sums with formula for $dA_{k,\ell}/dt$ we have
\begin{align*}
Q_t & = \frac{\nu}{L}[xR_x - y Q_y] + \frac{\nu}{L}[xQ_y - y Q_y] + \frac{\nu}{L}\alpha [xQ_y - y Q_y] + \frac{\nu}{L}\beta [yQ_x - y Q_x]  \\
& + ( 1 - \nu/L) [Q_y - y Q_y] + p ( 1 - \nu/L) [xQ_y - yQ_y]  \\
& +  ( 1 - \nu/L) \eta [xQ-Q] +  ( 1 - \nu/L) \eta [yQ-Q]
\end{align*}
Taking the terms from the last expression in the order 4, 1.1 (the first part of line 1), $2+3+6$, $1.2+5$, 7, and 8, we have
\begin{align}
Q_t & =  \frac{\nu}{L}\beta (y-x) Q_x  +  \frac{\nu}{L}xR_x 
\nonumber\\
& + \left( \left[ \frac{\nu}{L} + \alpha \frac{\nu}{L} + p ( 1 - \nu/L) \right] (x-y)  + (1-y) - \frac{\nu}{L} \right) Q_y 
\label{Qeqr1}\\
& + ( 1 - \nu/L) \eta (x-1) Q + ( 1 - \nu/L) \eta (y-1) Q 
\nonumber 
\end{align}
where we have set $\eta = N_{10}/N$, and recall $\gamma = N_{10}/N_1$ 

\clearp

\section{Moment equations} \label{sec:momeq}

From the differential equations of $Q$, and $R$, namely \eqref{Qeqr}-\eqref{Reqr}, in Section \ref{sec:Limit as Ltoinfty}, we obtained differential equations for the limits $U$, and $V$ (see \eqref{Ueqr}-\eqref{Veqr}), where
\begin{align*}
U(a,b)&=\lim_{L \rightarrow \infty} Q(1+a/L,1+b/L), \\
V(a,b)&= \lim_{L \rightarrow \infty} R(1+a/L,1+b/L).
\end{align*}
In Section \ref{sec:phalf}, in the symmetric case, we then looked for solutions of the form 
\beq
U(a,b) = \sum_{m,n} c_{m,n} a^m b^n, \qquad V(a,b) = \sum_{m,n} c_{n,m} a^m b^n.
\nonumber
\eeq
In this section, we provide the derivation of the moment equations of Section \ref{sec:phalf}. To this end, letting $r_1=1/2+\nu\bar\alpha$, $r_2=3/2+\nu\bar\alpha$, it follows from \eqref{Ueqr} that we need
\begin{align*}
0  = \nu\bar \beta &\sum_{m,n} c_{m,n} m a^{m-1} b^n (b-a) + \nu \sum_{m,n} c_{n,m} m a^{m-1} b^n \\
& + \sum_{m,n} c_{m,n} n a^m b^{n-1} [r_1 a - r_2 b - \nu] 
+ \bar\eta \sum_{m,n} c_{m,n} a^{m+1} b^n + \bar\eta \sum_{m,n} c_{m,n} a^{m} b^{n+1}  \\
 = \nu\bar \beta & \sum_{m,n} c_{m,n} m (a^{m-1} b^{n+1} - a^m b^n) + \nu \sum_{m,n} c_{n,m} m a^{m-1} b^n \\
& + \sum_{m,n} c_{m,n} n [r_1 a^{m+1} b^{n-1} - r_2 a^m b^n  - \nu a^m b^{n-1}] \\
& + \bar{\eta} \sum_{m,n} c_{m,n} a^{m+1} b^n + \bar{\eta} \sum_{m,n} c_{m,n} a^{m} b^{n+1}.
\end{align*} 
The coefficient of $a^mb^n$ is
\begin{align*}
0 & = \nu\bar\beta c_{m+1,n-1} (m+1) - \nu\bar\beta c_{m,n} m  + \nu c_{n,m+1}(m+1) \\
& + r_1 c_{m-1,n+1}(n+1) -r_2 c_{m,n} n - \nu c_{m,n+1} (n+1) + \bar\eta c_{m-1,n}  + \bar\eta c_{m,n-1}.
\end{align*}
So rearranging and filling in the values of the $r_i$ we have \eqref{recr} 
\begin{align*}
0 & = \bar \eta c_{m-1,n} + \bar\eta c_{m,n-1} \\
& +  \nu\bar\beta c_{m+1,n-1} (m+1) - (\nu\bar\beta m +r_2n) c_{m,n}   + r_1 c_{m-1,n+1}(n+1) \\
& + \nu c_{n,m+1}(m+1) - \nu c_{m,n+1} (n+1). 
\end{align*}

\mn
{\bf First order.} Taking $m=1$, $n=0$ then $m=0$, $n=1$  in \eqref{recr} we have
\begin{align*}
0 & = \bar \eta c_{0,0} - \nu\bar\beta c_{1,0}   + (1/2+\nu\bar\alpha) c_{0,1}  + \nu c_{0,2} \cdot 2 - \nu c_{1,1},\\
0 & = \bar \eta c_{0,0} + \nu\bar\beta c_{1,0} - (3/2 +\nu\bar\alpha ) c_{0,1}   + \nu c_{1,1} - \nu c_{0,2} \cdot 2.
\end{align*}
Recalling $m!n! c_{m,n} = \partial_a^m \partial_b^n U(0,0)$ this becomes
\begin{align*}
0 & = \bar \eta U - \nu\bar\beta U_a   + (1/2+\nu\bar\alpha) U_b  + \nu U_{bb} - \nu U_{ab},\\
0 & = \bar \eta U + \nu\bar\beta U_a - (3/2 +\nu\bar\alpha ) U_b   + \nu U_{ab} - \nu U_{bb}.
\end{align*}
Since $p=1/2$, we get $U(0,0)=1/2$, and therefore when we add these equations we find
\beq
0 = 2\bar\eta U - U_b,
\label{etaUb}
\eeq
so in general $\bar\eta = U_b$. Using this  in the first equation  we have
$$
0 = - \nu\bar\beta U_a  + (1+\nu\bar\alpha) U_b  + \nu ( U_{bb} -  U_{ab} ).
$$

\mn
{\bf Second order.} Taking $m=2,1,0$ and $n=2-m$ in \eqref{recr} we get
\begin{align*}
0 & = \bar \eta c_{1,0} + 0 + 0 - \nu\bar\beta c_{2,0} \cdot 2  + 0  + (1/2+\nu\bar\alpha) c_{1,1} + \nu c_{0,3} \cdot 3 - \nu c_{2,1}, \\
0 & = \bar \eta c_{0,1} + \bar \eta c_{1,0} + \nu\bar\beta c_{2,0} \cdot 2 - (\nu\bar\beta  + 3/2 + \nu\bar\alpha) c_{1,1}   
+ (1/2+ \nu\bar\alpha) c_{0,2} \cdot 2,  \\
0 & = 0 + \bar \eta c_{0,1} + \nu\bar\beta c_{1,1} + 0 - (3/2 + \nu\bar\alpha) c_{0,2} \cdot 2 + 0  + \nu c_{2,1} - \nu c_{0,3} \cdot 3.
\end{align*}
Recalling again $m!n! c_{m,n} = \partial_a^m \partial_b^n U(0,0)$, 
this becomes
\begin{align*}
0 & = \bar \eta U_a - \nu\bar\beta U_{aa}    + (1/2+\nu\bar\alpha) U_{ab} + \nu U_{bbb}/2 - \nu U_{aab}/2, \\
0 & = \bar\eta/(U_a+U_b) + \nu\bar\beta U_{aa}  - (\nu\bar\beta  + 3/2 + \nu\bar\alpha) U_{ab}   
+ (1/2+ \nu\bar\alpha) U_{bb},  \\
0 & = \bar \eta U_b + \nu\bar\beta U_{ab} - (3/2 + \nu\bar\alpha) U_{bb}  + \nu U_{aab}/2 - \nu U_{bbb}/2.
\end{align*}
Rearranging gives
\begin{align*}
&(\nu/2)(U_{aab} -  U_{bbb})  = \bar \eta U_a + U_{ab}/2 + \nu (\bar\alpha U_{ab} - \bar\beta U_{aa} ),   
\nonumber\\
0 & = \bar \eta(U_a+ U_b) -3U_{ab}/2 + U_{bb}/2 + \nu[ \bar\alpha (U_{bb} - U_{ab}) ]  + \bar\beta (U_{aa} - U_{ab}) ], \\
&(\nu/2)(U_{bbb} - U_{bbb})  = \bar \eta U_b  - 3U_{bb}/2 + \nu[\bar\beta U_{ab} - \bar\alpha U_{bb}].
\nonumber
\end{align*}
The middle equation is \eqref{E21r}. If we add the equations we get 
$$
0 = 2\bar\eta (U_a+U_b) - U_{ab} - U_{bb}.
$$
Since $U_a+U_b = 1/2$ and $\bar\eta = U_b$ we get \eqref{E20r}
$$
U_{aa} + U_{ab} = U_b .
$$

\mn
{\bf Third order.} Taking $m=3,2,1,0$ and $n=3-m$ in \eqref{recr} we get
\begin{align*}
0 &= \bar\eta c_{2,0} + 0 + 0 - \nu \bar\beta c_{3,0} \cdot 3 \\
& + 0 + (1/2 + \nu\bar\alpha) c_{2,1} \cdot 1 + \nu c_{0,4} \cdot 4 - \nu c_{3,1} \cdot 1, \\
0 & = \bar\eta c_{1,1} + \bar\eta c_{2,0} + \nu\bar\beta c_{3,0} \cdot 3 - \nu\bar\beta c_{2,1} \cdot 2 \\
& - (3/2+ \nu\bar\alpha) c_{2,1} \cdot 1 + (1/2+\nu\bar\alpha) c_{1,2} \cdot 2 + \nu c_{1,3} \cdot 3 - \nu c_{2,2} \cdot 2, \\
0 & = \bar\eta c_{0,2} + \bar\eta c_{1,1} + \nu\bar\beta c_{2,1} \cdot 2 - \nu\bar\beta c_{1,2} \cdot 1 \\
& - (3/2+ \nu\bar\alpha) c_{1,2} \cdot 2 + (1/2+\nu\bar\alpha) c_{0,3} \cdot 3 + \nu c_{2,2} \cdot 2 - \nu c_{1,3} \cdot 3, \\
0 & = 0 + \bar\eta c_{0,2} + \nu\bar\beta c_{1,2} \cdot 1  + 0 \\
& - (3/2+ \nu\bar\alpha) c_{0,3} \cdot 3 + 0 + \nu c_{3,1} \cdot 1 - \nu c_{0,4} \cdot 4. 
\end{align*}
Recalling once more $m!n! c_{m,n} = \partial_a^m \partial_b^n U(0,0)$,
\begin{align*}
0 &= \bar\eta U_{aa}/2 - \nu \bar\beta U_{aaa}/2 
 +  (1/2 + \nu\bar\alpha) U_{aab}/2 + \nu U_{bbbb}/3!  - \nu U_{aaab}/3!,\\
0 & = \bar\eta U_{ab} + \bar\eta U_{aa}/2 + \nu\bar\beta U_{aaa}/2 - \nu\bar\beta U_{aab} \\
& - (3/2+ \nu\bar\alpha) U_{aab}/2 + (1/2+\nu\bar\alpha) U_{abb} + \nu U_{abbb}/2  - \nu U_{aabb} / 2, \\
0 & = \bar\eta U_{bb}/2 + \bar\eta U_{ab} + \nu\bar\beta U_{aab} - \nu\bar\beta U_{abb}/2 \\
& - (3/2+ \nu\bar\alpha) U_{abb} + (1/2+\nu\bar\alpha) U_{bbb}/2 + \nu U_{aabb}/2 - \nu U_{abbb}/2, \\
0 & = \bar\eta U_{bb}/2 + \nu\bar\beta U_{abb}/2  
 - (3/2+ \nu\bar\alpha) U_{bbb}/2 + 0 + \nu U_{aaab}/3! - \nu U_{bbbb}/3!. 
\end{align*}
Rearranging gives
\begin{align}
(\nu/3!)( {U_{aaab}  - U_{bbbb}}) &= \bar\eta U_{aa}/2 - \nu \bar\beta U_{aaa}/2  +  (1/2 + \nu\bar\alpha) {U_{aab}}/2, 
\nonumber\\
(\nu/2)( {U_{aabb}  - U_{abbb}})  & = \bar\eta U_{ab} + \bar\eta U_{aa}/2 
\nonumber\\
& + \nu\bar\beta U_{aaa}/2 - \nu\bar\beta {U_{aab}} {-} (3/2+ \nu\bar\alpha) U_{aab}/2 + (1/2+\nu\bar\alpha) {U_{abb}},  
\nonumber\\
(\nu/2)( {U_{abbb} - U_{aabb}}) & = \bar\eta U_{bb}/2 + \bar\eta U_{ab} 
\label{order3}\\
&+ \nu\bar\beta U_{aab} - \nu\bar\beta U_{abb}/2  {-} (3/2+ \nu\bar\alpha) U_{abb} + (1/2+\nu\bar\alpha) U_{bbb}/2, 
\nonumber\\
(\nu/3!)( {U_{bbbb} - U_{aaab}}) & = \bar\eta U_{bb}/2 + \nu\bar\beta U_{abb}/2 
 - (3/2+ \nu\bar\alpha) U_{bbb}/2.  \nonumber
\end{align}

If we add all the equations, the voter terms cancel out and  we conclude that 
$$
0 = \bar\eta( U_{aa} + 2 U_{ab} + U_{bb} ) - U_{aab}/2 { - } U_{abb} - U_{bbb}/2. 
$$
If we use the fact that the first equation plus the fourth is 0 then we get a second new equation without the fourth
order variables. We have another unused second order equation but we have three new equations for four new unknowns
so we are falling further behind. 

\clearp

\section{Proofs of Theorems \ref{NL1i} and \ref{NL2}} \label{sec:BSth}
In this section we prove Theorem \ref{NL1i} and Theorem \ref{NL2}, which are generalizations of Basu and Sly results for the thick graphs. While proving these two theorems we follow the efficient algorithm, that is at each step we pick a discordant edge. Now let $\{Z_i\}$ be independent Bernoulli(1/2). If $Z_i=1$ we pick the end with 1 to be the 
left end point of the oriented edge; if $Z_i=0$ we pick the end with 0. 
We will have a collection of counters $K(v,m)$ to decide what type of event occurs (rewiring or voting). 
Suppose that on the $m$th step we decide to  update the oriented edge $(v,u)$. 
If $K(v,m-1)=0$ then $v$ imitates $u$. If $K(v,m-1)>0$ then $K(v,m)=K(v,m-1)-1$. We set
$K(x,m)=K(x,m-1)$ for $x \neq v$.

To create and update these counters, 
let $X_i$ and $X_i'$ be independent geometric($\nu/L$), taking values in $\{0, 1, 2 \ldots \}$. 
In addition we have two sequences of indices $I_i$ and $I'_j$ that start with $I_0=0$ and $I'_0=N$, i.e., we use the second sequence to
initialize the counters $K(i,0) = X'_i$. 
Let $T_0=0$ and let $W_k$ be i.i.d.~uniform on the set of vertices $V$. Recall that the set $S$ is the collection of vertices with initial degree less than equal to $11L$.

\begin{itemize}
  \item 
If $v\in S$ and $v$ is in state 0, set $I'_m=I'_{m-1}$. If $K(m-1,v)>0$ define $T_m = \min\{ k > T_{m-1} : W_k \neq v, W_k \not\sim v \}$.
Rewire to $W_{T_m}$. If $K(v,m-1)=0$ $v$ imitates $u$. Let $I_m = I_{m-1}+1$, $K(v,m) = X_{I_m}$.

\item
If $v\not\in S$ or $v$ is in state 1, set $I_m=I_{m-1}$. If $K(m-1,v)>0$ define $T_m = \min\{ k > T_{m-1} : W_k \neq v, W_k \not\sim v \}$.
Rewire to $W_{T_m}$. If $K(v,m-1)=0$ $v$ imitates $u$. Let $I'_m = I'_{m-1}+1$, $K(v,m) = X_{I'_m}$.
\end{itemize}

\begin{lemma} \label{Dmax}
Suppose that initially all vertices have degree $L$. Let $D_{max}(m)$ be the maximum degree after $m$ updates.
Let $\ep>0$ and $t>0$. While $D_{max}(tNL) \le CL$, we have $D_{max}(tNL) \le (1+\ep+t)L$ with high probability.
\end{lemma}

\begin{proof}
While $D_{max}(tNL) \le CL$ the number of values excluded by the conditions $ W_k \neq v$ and $W_k \not\sim v$ is $\le 1+CL$. Thus $T_m-T_{m-1}$ is stochastically dominated by a geometric random variable with success probability $1-\frac{1+CL}{N}$. Since $L=N^a$, with $a \in (0,1)$, using standard large deviation arguments it follows that if $L$ is large then $T_{tNL} \le (t+\ep/2)NL$ with high probability, and hence 
$$
\sup_v |\{ k \le (t+\ep/2)NL : W_k = v \}| \le (t+\ep) L.
$$
From this the desired result follows immediately.
\end{proof}

Now we are ready to prove Theorem \ref{NL1i}. Before going to the proof, let us first recall its statement once again.

\mn
{\bf Theorem \ref{NL1i}.}
{\it  Suppose $p \le 1/2$ and let $\ep>0$. If $\nu \le (0.15)/(1+3p)$ then with high probability $\tau < 3pNL$, and
the fraction of vertices in state 1 at time $\tau$ is between $p-\ep$ and $p+\ep$ with high probability.}

\begin{proof}
Let $X_m$ be the number of discordant edges after $m$ updates. Every time a rewiring occurs $X_m$ decreases by 1 with probability 1/2,
and stays the same with probability 1/2. After a voting event
$X_m$ can increase by at most $D_{max}(m)$. If $\tau > m$ then
\begin{align*}
E ( \exp(\lambda X_{m+1}/L) | {\cal F}_m ) \le \exp(\lambda X_m/L ) & \biggl[ (1-\frac{\nu}{L}) \left( 1+ \frac{1}{2} (e^{-\lambda/L} - 1) \right) \\
& + \frac{\nu}{L} \exp(\lambda D_{max}(m)/L) \biggr]
\end{align*}
Let $\tau_1$ be the first time $D_{max}(m) > (2.5+\ep)L$. It follows from Lemma \ref{Dmax} that $\tau_1> 1.5NL$ with high probability.
Let $\tau_0 = \min\{\tau,\tau_1\}$. When $\tau_0>m$ the quantity in square brackets is
$$
 \le  \left(1-\frac{\nu}{L}\right) \left( 1+ \frac{1}{2} (e^{-\lambda/L} - 1) \right)  + \frac{\nu}{L} \exp((2.5+\ep)\lambda)
$$
If $\lambda< \lambda_0(\ep)$ then $e^{(2.5+\ep)\lambda} \le 1 + (2.5+2\ep)\lambda$ and $e^{-\lambda/n} - 1 \le -(1-2\ep)\lambda/n$ then the above is
\begin{align}
& \le  \left(1-\frac{\nu}{L}\right) \left( 1 - (0.5-\ep) \frac{\lambda}{L} \right) + \frac{\nu}{L} (1+(2.5+2\ep)\lambda) 
\label{fineq}\\
& = 1 + (2.5+2\ep)\lambda \frac{\nu}{L} - (0.5-\ep) \frac{\lambda}{L} \left(1-\frac{\nu}{L}\right) 
\nonumber \\
& = 1 + (2.5\nu - 0.5(1-\nu/L)) + \ep (2\nu + (1-\nu/L)) \lambda/L
\nonumber
\end{align}

We first prove the result when $p=1/2$.
If we take $\nu = 3/50$ then $2.5\nu - 0.5 = -0.35$. If $\ep_0$ is small and $\lambda < \lambda_0(\ep_0)$ then for large $L$ the above is
$\le 1 - 0.34\lambda/L \le \exp( - 0.34\lambda/L)$. It follows that 
$$
P( \tau_0 > m ) \le E( \exp( \lambda X_m/n) 1_{\{\tau_0 > m\} } ) \le \exp( \lambda (X_0/L) - 0.34 m \lambda/L )
$$
Since there are $NL/2$ edges and each is discordant with probability $1/2$ we have $X_0 \le NL/2$
with high probability. Taking $m=(3/2)NL$
$$
P( \tau_0 > (3/2)NL ) \le \exp(  -0.02 \lambda N/2).
$$

To prove the result for $p<1/2$, note that if we only go out to time $3pNL$, \eqref{fineq} becomes
\begin{align*}
& \le  \left(1-\frac{\nu}{L}\right) \left( 1 - (0.5-\ep) \frac{\lambda}{L} \right) + \frac{\nu}{L} (1+(1+3p+2\ep)\lambda) 
\label{fineq}\\
& = 1 + (1+3p+2\ep)\lambda \frac{\nu}{L} - (0.5-\ep) \frac{\lambda}{L} \left(1-\frac{\nu}{L}\right)
\nonumber
\end{align*}
Let $\nu = 0.15/(1+3p)$. If $\ep$ is small enough then for large $L$ the above is
$\le 1 - 0.34p\lambda/L$. It follows that for $m \le 3pNL$
$$
P( \tau_0 > m ) \le E( \exp( \lambda X_m/n) 1_{\{\tau_0 > m\} } ) \le \exp( \lambda (X_0/L) - 0.34 m \lambda/L )
$$
Since there are $NL/2$ edges and each is discordant with probability $2p(1-p)$ we have $X_0 \le pNL$
with high probability. Taking $m=3pNL$
$$
P( \tau_0 > 3pNL ) \le \exp(  -0.02 \lambda pN/2).
$$
which completes the proof of Theorem \ref{NL1i}
\end{proof}

\medskip
Recall ${\cal G}(p,N,L) =$ graphs with vertex set $V = \{ 1, 2, \ldots, N \}$ labeled with 1's and 0's so that $N_1(G)=pn$ and 
the number of edges is $NL/2$. We now prove Theorem \ref{NL2}, after recalling its statement. 

\mn
{\bf Theorem \ref{NL2}.} {\it Let $\nu>0$ and $\ep>0$
There is a $p(\nu) < 1/2$ so that for all $G(0) \in {\cal G}(p,N,L)$ with $p \le p(\nu)$, we have with high probability
$\tau < 7NL$  and the fraction of vertices in state 1 at time $\tau$ is $\in (p-\ep,p+\ep)$.}

\mn
\begin{lemma} \label{densest}
If $G(0) \in {\cal G}(p,N,L)$ and $N$ is large enough then with high probability
the number of vertices in state 1 remains between $(p-\ep)$ and $(p+\ep)n$ throughout the first $7NL$ updates.
\end{lemma} 

\begin{proof}
On voting events the number of 1's changes by $\pm 1$ with equal probability independent of the past.
The expected number of voting steps is $7\nu N$ and with high probability will be $\le 8\nu N$. 
\end{proof}

\mn
Let $S$ be the set of vertices that at time 0 have degree $\le 11L$ and $T = V - S$. It follows that $|S| \ge 0.9N$. If not then $|T| \ge 0.1 N$,
and hence the number of edges in the graph is $\ge (0.1 N)(22L)/2 = 1.1NL/2$ contradicting our assumption that there are total $NL/2$ edges.

\begin{lemma} \label{Yest}
Call an $X_i$ stubborn if $X_i > 20L$. 
Let $Y = |\{ i \le L_{7NL} : X_i > 20L \}|$ be the number of stubborn elements which are used in
the first $7NL$ steps. Then with high probability $N_1(7NL) \ge Y$.
\end{lemma}

\begin{proof}
Note that stubborn $X$'s are used after a $v\in S$ in state 0 has flipped to 1. Let
$$
H = \{ | \{ i \le T_{7NL} : W_i = v \}| \le 8 L \hbox{ for all $v\in V$} \}
$$
By Lemma \ref{Dmax}, $H$ has high probability if $N$ is large.
On $H$ the number of rewirings to $v$ is $\le 8L$. If $v \in S$ then the initial degree $\le 11L$, so
the number of vertices that are ever connected to $v$ is $\le 19L$. If the number of rewiring events that are rooted 
at $v$ is $> 19n$ then we would run out of edges. This implies that if $N$ large then with high probability 
he number of events rooted at $v$ is $\le 20L$. Thus when a stubborn element is used the vertex will stay in
state 1 until time $7NL$.
\end{proof}

\begin{lemma} \label{RLSS}
Let $RL_{SS}$ denote the number of times a relabeling occurs when an edge with both endpoints
in $S$ is chosen. For $p$ sufficiently small $RL_{SS} \le \nu N/20$ with high probability.
\end{lemma}

\begin{proof} 
Let $RL^+_{SS}$ be the number that result in a change from 0 to 1. 
$$
RL^+_{SS} \le \min_i \{ \{j \le i : X_j \ge 20L \} > 1.1pN \}
$$
or using Lemma \ref{Yest} we have a contradiction of  Lemma \ref{densest}.
$P( X_j \ge 20L ) = (1-\nu/L)^{20L} \ge e^{-26\nu}$ if $L$ is large. It follows that if 
\beq
1.2p  < \frac{\nu}{50} e^{-26\nu}
\label{betaval}
\eeq
then with high probability there
are more than $1.1pN$ stubborn elements within the first $\nu N/50$ and hence $RL^+_{SS}$.

Each time a relabeling occurs it is equally likely to be $1 \to 0$ or $0 \to 1$, and these events are independent of each other, so
$$
P\left( RL_{SS} > \nu N/ 20, RL^+_{SS} \le \nu N/50 \right)  \le e^{-c\nu N/50} 
$$
The form of the right-hand side comes from the fact that the event is that fewer than 40\% of the first $\nu N/50$ have $Z_i=1$.
\end{proof}

\begin{lemma} \label{RSS}
Let $R_{SS}$ be the number of times an edge with both endpoints in $S$ was picked. For $p$
sufficiently small $R_{SS} \le NL/10$ with high probability.
\end{lemma}

\begin{proof} Each time an edge is picked, it leads to a relabeling with probability $\nu/L$. The result now follows 
from Lemma \ref{RLSS} and 
$$
P( R_{SS} > NL/10, RL_{SS} \le \nu N/20 ) \le P( Binomial(NL/10,\nu/L) \le \nu N/20 ) \le e^{-\nu N/80}
$$
To get the Binomial large deviations we use the fact that if $X=\hbox{Binomial}(M,p)$ then
$$
P( X \le M(p-z)) \le \exp(-Nz^2/2p)
$$
See e.g., \cite[Lemma 2.8.5]{RGD}. Here $M=NL/10$, $p=\nu/L$ $Mp=\nu N/10$, $Mz=\nu N/20$, $z=\nu/2L$, so we have
$$
\frac{Nz^2}{p} = \frac{NL}{10} \cdot \frac{\nu^2}{4L^2} \cdot \frac{L}{2\nu} = \frac{\nu N}{80}
$$
\end{proof}

\begin{lemma} \label{RST}
Let $R_{ST}$ be the number of times a disagreeing edge was picked with one endpoint in $S$ and the other in $T$.
For $p$ sufficiently small $R_{ST} \le 2.8NL$ with high probability.
\end{lemma}

\begin{proof} Let $W_{ST}$ be the number of rewiring moves  with one endpoint in $S$ and the other in $T$. On each of these moves
1/2 the time it is rewired with the root at $S$ and if $n$ is large then with probability at least
$$
\frac{|S|-1}{n-1} \ge \frac{8}{9} 
$$
the new vertex is in $S$. Let $Y_{SS}$ be the number of $S$ to $S$ edges at the end and $W_{ST\to SS}$ be the number of $(S,T)$ to $(S,S)$ rewirings.
We must have
$$
NL \ge Y_{SS} \ge W_{ST\to SS} - R_{SS}
$$
so if $R_{SS} \le NL/10$ (which has high probability by Lemma 6.5) we must have $W_{ST\to SS} \le 1.1NL$. If $W_{ST} \ge 2.7NL$
then we expect $W_{ST \to SS} \ge (2.7NL)(4/9) = 1.2NL$ so 
$$
P(  W_{ST\to SS} \le 1.1NL , W_{ST} \ge 2.7NL ) \le exp(-cNL)
$$
Since each time a disagreeing edge is picked, with probability $1-\nu/L$ it leads to a rewiring,
$$
P( R_{ST} > 2.8NL, W_{ST} \le 2.7NL ) \le e^{-cNL}
$$
which completes the proof. \end{proof}

\begin{lemma} \label{RTT}
Let $R_{TT}$ be the number of times a disagreeing edge was picked with both endpoints in $T$.
For $p$ sufficiently small $R_{TT} \le 4NL$ with high probability.
\end{lemma}

\begin{proof} From the proof of Lemma \ref{RST}, we see that after rewiring an edge with both endpoints in $T$ the chance is becomes
an edge with one edge in $S$ and one in $T$ is $\ge 9/10$. Let $Y_{ST}$ be the number of $S$ to $T$ edges at the end and 
$W_{TT\to ST}$ be the number of $(T,T)$ to $(S,T)$ rewirings. We must have
$$
NL \ge Y_{ST} \ge W_{TT\to ST} - R_{ST}
$$
so if $R_{ST} \le 2.8NL$ (which has high probability by Lemma \ref{RST}) we must have $W_{TT\to ST} \le 3.8NL$.
Arguing as in the previous lemma we can conclude that with high probability $W_{TT} \le 3.9NL$ and $R_{TT} \le 4NL$.
\end{proof}

\mn
{\bf Proof of Theorem 6.} From Lemmas \ref{RSS}, \ref{RST}, and \ref{RTT} we see that for $p <p(\nu)$ and $G(0) \in {\cal G}(p,N,L)$
we have $R_{SS} + R_{ST} + R_{TT} \le 6.9 NL$ with high probability so there are no discordant edges after $7NL$ updates.

\clearp

\end{document}